\documentclass{article}

     \usepackage[nonatbib,preprint]{neurips_2020}

\usepackage[utf8]{inputenc} 
\usepackage[T1]{fontenc}    
\usepackage{hyperref}       
\usepackage{url}            
\usepackage{booktabs}       
\usepackage{amsfonts}       
\usepackage{nicefrac}       
\usepackage{microtype}      

\usepackage[square,numbers]{natbib}
\usepackage{amsthm}
\usepackage{microtype}
\usepackage{graphicx}

\usepackage{amssymb}
\usepackage{enumitem}

\usepackage{booktabs} 
\usepackage{subcaption}
\usepackage{epstopdf} 
\usepackage{amsmath}
\usepackage{wrapfig}

\usepackage[final]{changes}
\newtheorem{lemma}{Lemma}
\newtheorem{theorem}{Theorem}
\newtheorem{proposition}{Proposition}

\newcommand{\sign}{\texttt{sign}}
\newcommand{\R}{\mathbb{R}}
\usepackage{hyperref}

\title{A flexible framework for 
communication-efficient machine learning: from HPC to IoT}

\author{%
   Sarit~Khirirat \\
      Division of Decision and Control Systems\\
      KTH Royal Institute of Technology\\
      \texttt{sarit@kth.se} \\
       \And
       Sindri~Magnússon \\
      Division of Decision and Control Systems\\
      KTH Royal Institute of Technology\\
      \texttt{sindrim@kth.se} \\
      \And 
      Arda~Aytekin \\ 
      Ericsson  \\ 
      Stockholm, Sweden \\ 
      \texttt{arda@aytekin.biz}\\
       \And
       Mikael~Johansson  \\
      Division of Decision and Control Systems\\
      KTH Royal Institute of Technology\\
      \texttt{mikaelj@kth.se} 
}

\begin{document}

\maketitle

\begin{abstract}
With the increasing scale of machine learning tasks, it has become essential to reduce the communication between computing nodes. Early work on gradient compression focused on the bottleneck between CPUs and GPUs, but  communication-efficiency is now needed in a variety of  different system architectures, from high-performance clusters to energy-constrained IoT devices. In the current practice, compression levels are typically chosen before training and settings that work well for one task may be vastly suboptimal for another dataset on another architecture. In this paper, we propose a flexible framework which adapts the compression level to the true gradient at each iteration, maximizing the improvement in the objective function that is achieved per communicated bit. Our framework is easy to adapt from one technology to the next by modeling how the communication cost depends on the compression level for the specific technology. Theoretical results and practical experiments indicate that the automatic tuning strategies significantly increase communication efficiency on several state-of-the-art compression schemes.  
\end{abstract}

\section{Introduction}

 The vast size of modern machine learning is shifting the focus for optimization and learning algorithms from centralized to distributed architectures. State-of-the-art models are now typically trained using multiple CPUs or GPUs and  data is increasingly being collected and processed in networks of
resource-constrained devices, e.g., IoT devices, smart phones, or wireless
sensors. This trend is shifting the bottleneck  from the computation to the communication. 
The shift is particularly striking when learning is performed on energy-constrained devices that communicate over shared wireless channels. 
 %
%
%
 %
 Indeed, distributed training is often communication bound since the associated   
  optimization algorithms hinge on frequent transmissions of gradients between nodes. These gradients are typically huge:  it is not uncommon for state-of-the-art models to have millions of parameters. 
 To get a sense of the corresponding communication cost, transmitting a single gradient or stochastic gradient using single precision (32 bits per entry) requires 40 MB for a model with 10 million parameters.   
If we use 4G communications, this means that we can expect to transmit roughly one gradient per second. 
 The huge communication load easily overburdens training even on loosely interconnected clusters and may render federated learning on some IoT or edge devices infeasible.


 To alleviate the communication bottleneck, much recent research has focused on compressed gradient methods. These methods achieve communication efficiency by using only the most informative parts of the gradients at each iteration. We may, for example, sparsify the gradient, \emph{i.e.} use only the most significant entries at each iteration and set the rest to be zero~ \cite{alistarh2017qsgd,alistarh2018convergence,stich2018sparsified,wen2017terngrad,wang2018atomo,khirirat2018gradient,wangni2018gradient}. 
 We may also quantize the gradient elements or  do some mix of quantization and sparsification~\cite{alistarh2017qsgd,khirirat2018distributed,magnusson2017convergence,wangni2018gradient,zhu2016trained,rabbat2005quantized}. Several of the cited papers have demonstrated huge communication improvements for specific training problems.  However, these communication benefits are often realized after a careful tuning of the compression level before training, \emph{e.g.} the number of elements to keep when sparsifying the gradient. We cannot expect there to be a universally good compressor that works well on all problems, as shown by the worst-case communication complexity of optimization in~\cite{tsitsiklis1987communication}.  There is generally a delicate problem-specific balance between compressing too much or too little.  Trying to strike the right balance by  hyper-parameter tuning is expensive and the resulting tuning parameters will be problem specific. Moreover, most existing compression schemes are agnostic of the disparate communication costs for different technologies. In contrast, our proposed on-line mechanism adapts the compression level to the information content of each gradient and the platform-specific communication cost.

\textbf{Contributions:} We propose a flexible framework for on-line adaption of the gradient compression level to the problem data and communication technology used. This Communication-aware Adaptive Tuning (CAT) optimally adjusts the compression of each communicated gradient by maximizing ratio between guaranteed objective function improvement and communication cost.
%
%
 The communication cost can easily be adjusted to the technology used, either by analytical models of the communication protocols or through empirical measurements. 
We illustrate these ideas on three state-of-the-art compression schemes: a) sparsification, b) sparsification with quantization and c) stochastic sparsification. In all cases, we first derive descent lemmas specific to the compression, relating  the function improvement to the  tuning parameter. Using these results we can find the tuning that optimizes the communication efficiency measured in descent relative to the given communication costs.
%
 Even though most of our theoretical results are for a single node, we illustrate the efficiency of CAT to all three compression schemes in large scale experiments in multi-nodes settings. 
  For the stochastic sparsification we also prove convergence for stochastic gradient in multi-node settings.

\textbf{Notation:}
We let $\mathbb{N}$, $\mathbb{N}_0$, and $\mathbb{R}$ be the set of natural numbers, the set of natural numbers including zero, and the set of real numbers, respectively.  
The set $\{a,a+1,\ldots, b\}$ is denoted by $[a,b]$ where $a,b\in\mathbb{N}_0$ and $a\leq b$.
For $x\in\mathbb{R}^d$, $\|x\|_q$  is the $\ell_q$ norm with $q \in (0,\infty]$. A continuously differentiable function $F:\mathbb{R}^d\rightarrow\mathbb{R}$ is $L$-smooth if for all $x,y\in\mathbb{R}^d$ we have
%
%
    $\| \nabla F(x) - \nabla F(y) \| \leq L \| x-y \|$.
We say that $F$ is $\mu$-strongly convex if  
   $ F(y) \geq F(x) + \langle \nabla F(x),y-x \rangle + ({\mu}/{2})\| y-x\|^2.$

\section{Background}

 The main focus of this paper is empirical risk minimization 
 $$\underset{x\in \R^d}{\text{minimize }} F(x)= \frac{1}{|\mathcal{D}|}\sum_{z\in\mathcal{D}} L(x;z)$$
 where $\mathcal{D}$ is a set of data points\added{, $x$ is the model parameters} and $L(\cdot)$ is a loss function.

 \subsection{Gradient Compression } \label{Sec:GC}

\replaced{Consider the standard compressed gradient iteration}{We study compressed gradient methods} 
 \begin{equation} \label{Alg:main} x^{i+1}= x^i-\gamma^i Q_T(\nabla F(x^i)).\end{equation}
  Here, $Q_T(\cdot)$ is a compression operator  and $T$ is a parameter that controls the level of compression. 
 The goal  is to  achieve communication efficiency by using only the most significant information.
 One of the simplest compression schemes is to sparsify the gradient, \emph{i.e.} to let 
\deleted{In this case $Q(\cdot)$ is the sparsifing operator} 
\begin{equation} \label{eq:T_sparse_operator}  [Q_T(g)]_j=\begin{cases} g_j & \text{if }j\in I_T(g) \\ 0 & \text{otherwise.} \end{cases}
\end{equation}
where $I_T(g)$ is the index set for the $T$ components of $g$ with largest magnitude.
 The following combination of sparsification and quantization has been shown to give good practical performance~\cite{alistarh2017qsgd}: 
\begin{equation}  \label{eq:T_sparse_quant} 
 [Q_{T}(g)]_j=\begin{cases}||g|| \sign(g_j) & \text{if } i \in I_T(g) \\ 0 & \text{otherwise}. \end{cases}
\end{equation}
 In this case, we communicate only the gradient magnitude and the sparsity pattern of the gradient. 
\
 It is sometimes advantageous to use stochastic sparsification. 
 %
 %
 %
 %
 Rather than sending the top $T$ entries of each gradient, we then send $T$ components on average. 
We can achieve this by setting
\begin{align} \label{eq:StochSpar}
    [{Q}_{T,p} \left( g \right)]_j = ({g_j}/{p_j})\xi_j , 
\end{align}
where $\xi_j \sim \texttt{Bernouli}(p_j)$ and 
\(T = \sum_{j=1}^d p_j.\)
 Ideally, we would like $p_j$ to represent the magnitude of $g_j$, so that if $|g_j|$ is large relative to the other entries then $p_j$ should also be large. 
 There are many heuristic methods to choose $p_j$.  For example, if we set $p_j = |g_j|/\|  g \|_q$ with $q=2$, $q=\infty$, and $q\in(0,\infty]$ then we get, respectively, the stochastic sparsifications in~\cite{alistarh2017qsgd} with $s=1$, the \emph{TernGrad} in \cite{wen2017terngrad}, and $\ell_q$-quantization in \cite{wang2018atomo}.
 We can also optimize $p$, see~\cite{wang2018atomo} and Section~\ref{sec:DynSS} for  details.


  Experimental results have shown  huge communication savings by
  compressed gradient methods in large-scale machine learning~\cite{shi2019distributed,ijcai2019-473}.  
  Nevertheless, we can easily create pedagogical examples where they are no more communication efficient than full gradient descent. 
 For sparsification, consider the function
 $F(x) = ||x||^2/2$.
 Then, gradient descent with step-size $\gamma=1$ converges in one iteration, and thus communicates only $d$ floating points (one for each element of $\nabla F(x)\in \mathbb{R}^d$)  to reach any $\epsilon$-accuracy.  
 On the other hand, $T$-sparsified gradient descent (where $T$ divides $d$) needs $d/T$ iterations, which implies $d$ communicated floating points in total. In fact, the sparsified method is even worse since it requires an additional  $d\log(d)$ communicated bits to indicate the sparsity pattern. 
 
  This example shows that the benefits of sparsification are not seen on worst-case problems, and that traditional worst-case analysis (\emph{e.g.}~\citet{khirirat2018gradient}) is unable to guarantee improved communication complexity. Rather, sparsification is useful for exploiting structure that appears in real-world problems. 
  The key in exploiting this structure is to choose $T$ properly at each iteration. 
  In this paper 
  we describe how to choose $T$ dynamically to optimize the communication efficiency of sparsification.

\subsection{Communication Cost: Bits, Packets, Energy and Beyond}

 The compressors discussed above have a tuning parameter $T$, which controls the sparsity budget of the compressed gradient. 
 Our goal is to tune $T$ adaptively to optimize the communication efficiency. To explain how this is done, we first need to discuss how to model the cost of communication. 
 Let $C(T)$ denote the  communication cost per iteration as a function of $T$, e.g.,  total number of transmitted bits. 
%
 Then, $C(T)$ consists of payload (actual data) and communication overhead. 
 The payload is the number of bits required to communicate the compressed gradient. 
 For the sparsification  in Eq.~\eqref{eq:T_sparse_operator} and the quantization in Eq.~\eqref{eq:T_sparse_quant}, the payload consumes, respectively, 
   \begin{equation}  \label{eq:Payload_SQ}
    P^\texttt{S}(T)=T \times ( \lceil \log_2(d) \rceil + \texttt{FPP}) ~~ \text{bits} ~~\text{ and }~~ P^{\texttt{SQ}}(T) = \texttt{FPP}+ T \times \lceil \log_2(d) \rceil~~\text{ bits},
 \end{equation} 
 where the $\log_2(d)$ factor comes from indicating $T$ indices in the $d$ dimensional gradient vector. Here 
 $\texttt{FPP}$ is our floating point precision, e.g., $\texttt{FPP}=32$ or $\texttt{FPP}=64$ for, respectively, single or double precision floating-points.
 %
 %
 %
 %
 %
 \added{Our simplest communication model accounts only for the payload, i.e., $C(T)=P(T)$. 
 We call this the \emph{payload model}.}
%
%
%
In real-world networks, however, each communication also includes overhead and set-up costs. 
A more realistic model is therefore affine
     $C(T)=c_1 P(T)+c_0,$ 
 where $P(T)$ is the payload. Here $c_0$ is the communication overhead while $c_1$ is the cost of transmitting a single payload byte.  For example, if we just count transmitted bits ($c_1=1$), a single UDP packet transmitted over Ethernet requires an overhead of $c_0=54\times 8$ bits and can have a payload of up to $1472$ bytes. In the wireless standard IEEE 802.15.4, the overhead ranges from $23$-$82$ bytes, leaving $51-110$ bytes of payload before the maximum packet size of $133$ bytes is reached~\cite{KoS:17}. 
 Another possibility is to use a \emph{packet model}, \emph{i.e.} to have a fixed cost per packet
%
 %
 %
 %
\begin{equation}
    C(T)= c_1  \times \left\lceil {P(T)}/{P_{\max}} \right\rceil+ c_0,  \label{eq:COMM2}
\end{equation}
 where $P_{\max}$ is the number of payload bits per packet. The term $\lceil P(T)/P_{\max} \rceil$ counts the number of packets required to send the $P(T)$ payload bits, $c_1$ is the cost per packet, and $c_0$ is the cost of initiating the communication.  
  These are just two examples; ideally, $C(T)$ should be tailored to the specific communication standard used, and possibly even estimated from system measurements.

\subsection{Key Idea: Communication-aware Adaptive Tuning (CAT)}

 %
 
 When communicating the compressed gradients we would 
like to use each bit as efficiently as possible. 
 In optimization terms, we would like the  objective function improvement for each communicated bit to be as large as possible.
 In other words, we want to maximize the ratio
 \begin{equation} \texttt{Efficiency}(T)=\frac{\texttt{Improvement}(T)}{C(T)}, \label{eq:Efficiency}
 \end{equation}
 where $\texttt{Improvement}(T)$ is the improvement in the objective function when we use $T$-sparsification with the given compressor.
 We will demonstrate how the value of $\texttt{Improvement}(T)$ can be obtained from novel descent lemmas 
 and derive dynamic sparsification policies which, at each iteration, find the $T$ that optimizes $\texttt{Efficiency}(T)$. We derive optimal $T$-values for the three  compressors and two communication models introduced above. However, the idea is general and can be used to improve the communication efficiency for many other compression optimization algorithms.

\section{Dynamic Sparsification}
We begin by describing how our CAT framework applies to sparsified gradient methods. To this end, the following lemma introduces a useful measure $\alpha(T)$ of function value improvement: 
 \begin{lemma}\label{lemma:Descent_sparse}
   Suppose that $F:\R^d\rightarrow \R$ is (possibly non-convex) $L$-smooth and $\gamma=1/L$. Then for any $x,x^+\in\R^d$ with
    $x^+=x-\gamma Q_T(\nabla F(x))$
   we have
   $$ F(x^+)\leq F(x) - \frac{\alpha(T)}{2L} ||\nabla F(x)||^2 ~~~~~
   \text{ where } ~~~~~
   \alpha(T)=\frac{||Q_T(\nabla F(x))||^2}{||\nabla F(x) ||^2}.$$
  Moreover, there are $L$-smooth functions $F(\cdot)$ for which the inequality is tight for every $T=1,\ldots,d$.
 \end{lemma}
  This lemma 
  is in the category of descent lemmas, 
  which are  standard tools to study the convergence for convex and non-convex functions. In fact, Lemma~\ref{lemma:Descent_sparse} generalizes the standard descent lemma for $L$-smooth functions (see, e.g., in Proposition A.24 in~\cite{nonlinear_bertsekas}).
  In particular, 
%
  %
 if the gradient $\nabla F(x)$ is  $T$-sparse (or $T=d$) then Lemma~\ref{lemma:Descent_sparse} gives the standard descent
 $$ F(x^+)\leq F(x) - \frac{1}{2L} ||\nabla F(x)||^2.$$
 
 In the next subsection, we will use Lemma~\ref{lemma:Descent_sparse} to derive novel convergence rate bounds for sparsified gradient methods, extending many standard results for gradient descent. 
First, however, we will use $\texttt{Improvement}(T)=\alpha(T)$ to define the following CAT mechanism for dynamic sparsification:
\begin{align}
  \textbf{Step 1: }  T^i=\underset{T\in[1,d]}{\text{argmax}}~~ \frac{\alpha^i(T)}{C(T)} 
    ,~~~~~~~  \textbf{Step 2: }  x^{i+1}=x^i -\frac{1}{L} Q_{T^i}(\nabla F(x^i)). \label{eq:Alg1_G}
\end{align} 
 The algorithm first 
 finds the sparsity budget $T^i$ that optimizes the communication efficiency defined in (\ref{eq:Efficiency}), and then
performs a standard sparsification using this value of $T^i$. 
Since $\Vert \nabla F(x)\Vert^2/2L$ is independent of $T$, maximizing $\alpha(T)/C(T)$ also maximizes efficiency. 

To find $T^i$ at each iteration we need to solve the maximization problem in Eq.~\eqref{eq:Alg1_G}. This problem has one dimension, and even a brute force search would be feasible in many cases. As the next two results show, however, the problem has a favourable structure that allows the maximization to be solved very efficiently. 
The first result demonstrates that 
the descent always increases with $T$ and is bounded. \deleted{as follows.} 
 \begin{lemma} \label{Lemma:LB_alpha}
   For $g\in \R^d$  the function \(\alpha(T)={||Q_T(g)||^2}/{||g||^2}\)
   is increasing and concave when extended to the continuous interval $[0,d]$.
   Moreover, $\alpha(T)\geq {T}/{d}$ for all $T\in\{0,\ldots d\}$ and there
   exists an $L$-smooth function such that ${||Q_T(\nabla f(x))||^2}/{||\nabla f(x)||^2}=T/d$ for all $x\in\R^d$.
  %
  %
 \end{lemma}

Lemma \ref{Lemma:LB_alpha} results in many consequences in the next section, but first we make another observation:
\begin{proposition} \label{prop:alpha_over_c}
Let $\alpha(T)$ be increasing and concave. If $C(T)=\tilde{c}_1 T + c_0$, then $\alpha(T)/C(T)$ is quasi-concave and has a unique maximum on $[0,d]$. 
When $C(T)= \tilde{c}_1 \lceil T / \tau_{\max} \rceil + c_0 $, on the other hand, $\alpha(T)/C(T)$ attains its maximum for a $T$ which is an integer multiple of $\tau_{\max}$.
\end{proposition}

This proposition shows that the optimization in Eq.~\eqref{eq:Alg1_G} is easy to solve. For the affine communication model, one can simply sort the elements in the decreasing order, initialize $T=1$ and increase $T$ until $\alpha(T)/C(T)$ decreases. In the packet model, the search for the optimal $T$ is even more efficient, as one can increase $T$ in steps of $\tau_{\max}$.

    \renewcommand{\arraystretch}{1.4}
    \addtolength{\tabcolsep}{-0.9pt}
   \begin{table}[t]
\caption{Iteration complexity for $T$-sparsified (stochastic) gradient descent. We prove these results and analogue results for $S+Q$ compression in Appendix~\ref{app:Table1}.
  \textbf{Deterministic Sparsification:} $A_{\epsilon}^{\texttt{SC}}$,
  $A_{\epsilon}^{\texttt{C}}$, and $A_{\epsilon}^{\texttt{NC}}$ are standard upper
  bounds for gradient descent on the number of iterations needed to achieve
  $\epsilon$-accuracy for, respectively, strongly-convex, convex, and non-convex
  problems.
  In particular, $A_{\epsilon}^{\texttt{SC}} =
  \kappa\log(\epsilon_0/{\epsilon})$, $A_{\epsilon}^{\texttt{C}}=2LR^2/\epsilon$,
  $A_{\epsilon}^{\texttt{NC}}=2L\epsilon_0/\epsilon$, where $\kappa=L/\mu$, $\epsilon_0
  = F(x^0)-F^{\star}$, and $R$ is a constant such that $||x^i-x^{\star}|| \leq R$
  (for some optimizer $x^{\star}$).
  \textbf{Stochastic Sparsification:} $B_{\epsilon}^{\texttt{SC}}$,
  $B_{\epsilon}^{\texttt{C}}$, and $B_{\epsilon}^{\texttt{NC}}$ are standard upper
  bounds for multi-node stochastic gradient descent on the number of iterations
  needed to achieve $\epsilon$-accuracy (in expected value) for, respectively, strongly-convex, convex, and non-convex problems. In particular,
  $B_{\epsilon}^{\texttt{SC}} = 2 (1+2\sigma^2/(\mu\epsilon
  L))(A_{\epsilon}^{\texttt{SC}} + \delta)$, $B_{\epsilon}^{\texttt{C}}
  = 2 (1+2\sigma^2/(\epsilon L)) A_{\epsilon}^{\texttt{C}}$, and
  $B_{\epsilon}^{\texttt{NC}}= 2(1+2\sigma^2/\epsilon)A_{\epsilon}^{\texttt{NC}}$,
  where $\delta = \kappa \log(2)$ and $\sigma^2$ is a variance bound of stochastic
  gradients. The $\epsilon$-accuracy is measured in $\mathbf{E}[F(x)-F(x^\star)]$ for convex problems, and in $\mathbf{E}\| \nabla F(x) \|^2$ otherwise. 
  }
\vskip 0.15in
\begin{center}
\begin{small}
\begin{sc}
\begin{tabular}{lcccccc}
\toprule
& \multicolumn{3}{c}{Deterministic Sparsification} & \multicolumn{3}{c}{Stochastic Sparsification} \\
 Upper-Bound    & $\mu$-convex & convex & nonconvex     & $\mu$-convex & convex & nonconvex    \\
\midrule


No-Compression & $A_{\epsilon}^{\texttt{SC}}$ & $A_{\epsilon}^{\texttt{C}}$ & $A_{\epsilon}^{\texttt{NC}}$
& $B_{\epsilon}^{\texttt{SC}}$ & $B_{\epsilon}^{\texttt{C}}$ & $B_{\epsilon}^{\texttt{NC}}$

\\
Data-Dependent & $ \textcolor{blue}{\frac{1}{\bar{\alpha}_T}} A_{\epsilon}^{\texttt{SC}}$ &  $\textcolor{blue}{\frac{1}{\bar{\alpha}_T}}A_{\epsilon}^{\texttt{C}}$ & $\textcolor{blue}{\frac{1}{\bar{\alpha}_T}} A_{\epsilon}^{\texttt{NC}}$ & $\textcolor{blue}{\frac{1}{\bar{\omega}_T}} B_{\epsilon}^{\texttt{SC}}$ & $\textcolor{blue}{\frac{1}{\bar{\omega}_T}} B_{\epsilon}^{\texttt{C}}$ & $\textcolor{blue}{\frac{1}{\bar{\omega}_T}} B_{\epsilon}^{\texttt{NC}}$

\\

 Worst-Case
  &  $\textcolor{red}{\frac{d}{T}}A_{\epsilon}^{\texttt{SC}}$ & $\textcolor{red}{\frac{d}{T}}  A_{\epsilon}^{\texttt{C}}$  & $\textcolor{red}{\frac{d}{T}}  A_{\epsilon}^{\texttt{NC}}$  &
  $\textcolor{red}{\frac{d}{T}} B_{\epsilon}^{\texttt{SC}}$ & $\textcolor{red}{\frac{d}{T}}B_{\epsilon}^{\texttt{C}}$ & $\textcolor{red}{\frac{d}{T}} B_{\epsilon}^{\texttt{NC}}$
  \\

\bottomrule
\end{tabular}
\end{sc}
\end{small}
\label{Tab:SDG_IC}
\end{center}
\vskip -0.2in
\end{table} 

\subsection{Dynamic sparsification benefits in theory and practice}\label{subsec:benefitDynamicSparsifier}
 
 Although the CAT framework applies to general communication costs, it is
 instructive to see what our results say about the communication complexity,
 i.e., the number of bits that need to be communicated to guarantee that a solution is found with $\epsilon$-accuracy. Table~\ref{Tab:SDG_IC} compares the iteration complexity of Gradient Descent (GD) in row 1 and $T$-Sparsified Gradient Descent ($T$-SGD) in rows 2 and 3 with constant $T$ for strongly-convex, convex, and non-convex problems. The results for gradient descent are well-known and found in, e.g.,~\cite{nesterov2018lectures}, while the worst-case analysis is from~\cite{khirirat2018gradient}. 
 The results for $T$-sparsified gradient descent are derived using Lemma~\ref{lemma:Descent_sparse} instead of the standard descent lemma; see proofs in the supplementary.  
%
%
  \begin{figure*}
    \centering
    \begin{subfigure}[t]{0.3\textwidth}
        \includegraphics[width=\textwidth]{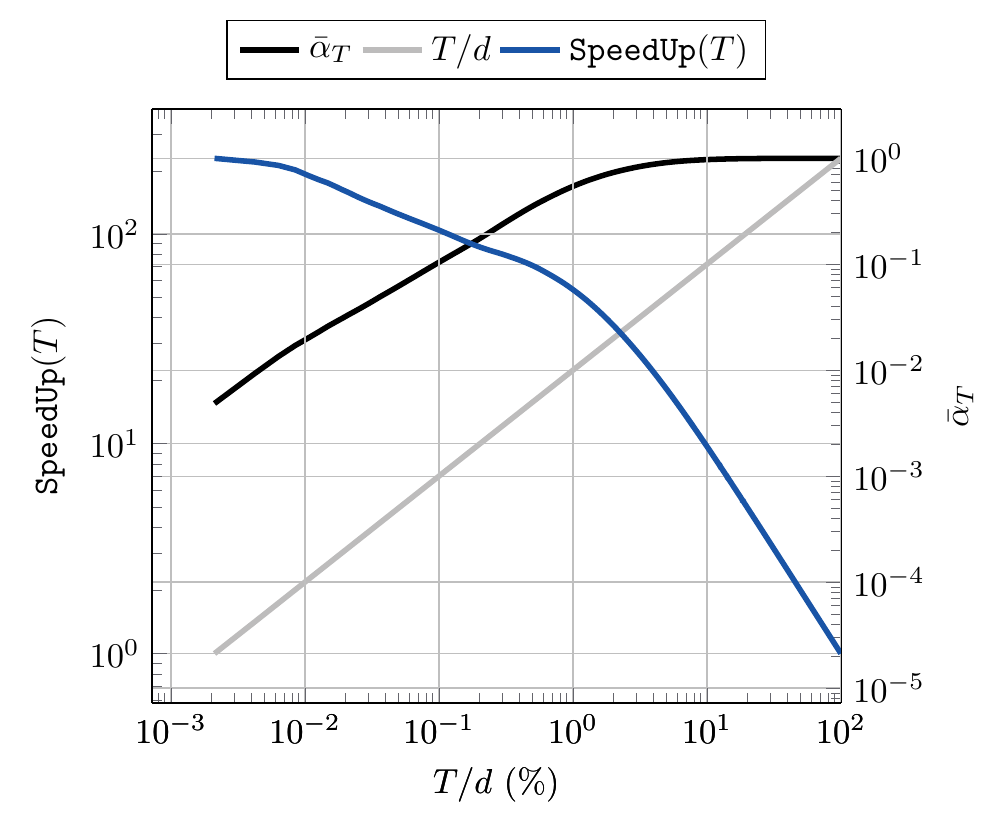}
        \caption{$\bar{\alpha}_T$ and speedup}
        \label{fig:SGD_Alpha}
    \end{subfigure}
    ~ 
          \begin{subfigure}[t]{0.3\textwidth}
        \includegraphics[width=\textwidth]{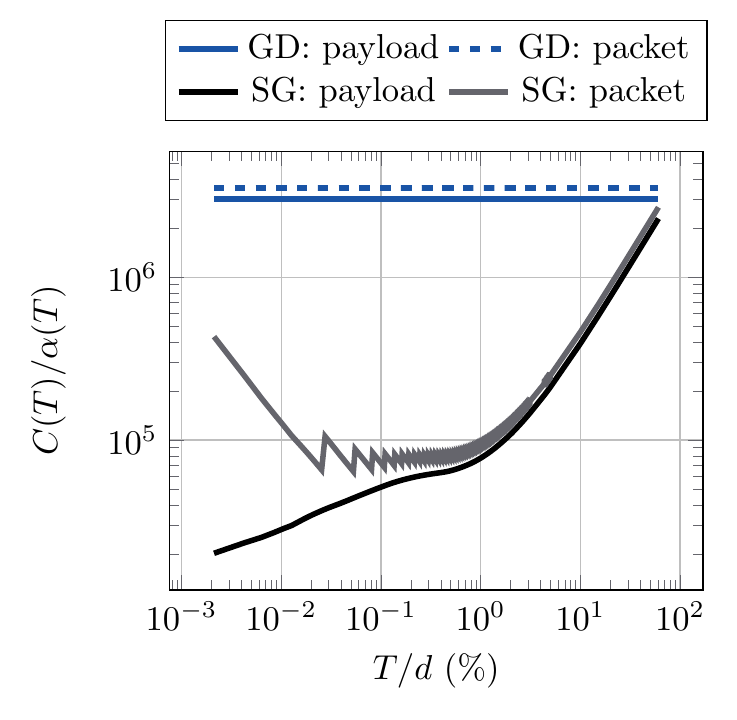}
        \caption{Hypothetical: Communications to reach $\epsilon$-accuracy}
        \label{fig:SGD_TCC}
        
    \end{subfigure}
    ~
    \begin{subfigure}[t]{0.3\textwidth}
        \includegraphics[width=\textwidth]{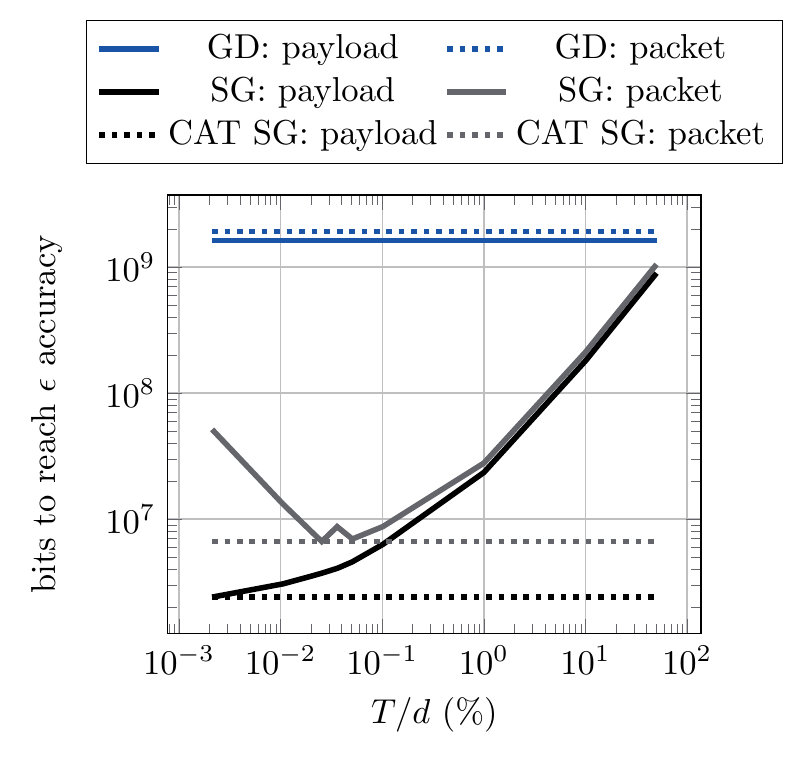}
        \caption{Experiments: Communications to reach $\epsilon$-accuracy}
        \label{fig:SGD_Experiments}
    \end{subfigure}
    \caption{CAT sparsified gradient descent on the \texttt{RCV1} data set. 
    }\label{fig:animals}
\end{figure*}
%
%
 %
 %
 Comparing rows 1 and 3 in the table, we see that the worst-case analysis does not guarantee any improvements in the amount of communicated floating points. Although $T$-SGD only communicates $T$ out of $d$ gradient entries in each round, we need to perform $d/T$ times more iterations with $T$-SGD than with SGD, so the two approaches will need to communicate the same number of floating points. In fact, $T$-SGD will be worse in terms of communicated bits since it requires $T\lceil \log_2(d) \rceil$ additional bits per iteration to indicate the sparsity pattern.  
  
 Let us now turn our attention to our novel analysis shown in row 2 of Table~\ref{Tab:SDG_IC}. Here, 
 the parameter $\bar{\alpha}_T$ is a lower bound on $\alpha^i(T)$ over every iteration, that is
 $$\alpha^i(T) \geq  \bar{\alpha}_T ~~\text{ for all } i. $$
 Unfortunately, $\bar{\alpha}_T$ is not useful for algorithm development: we know from Lemma~\ref{Lemma:LB_alpha} that it can be as low as $T/d$, and it is not easy to compute a tight data-dependent bound off-line, since $\bar{\alpha}_T$ depends on the iterates produced by the algorithm. However, $\bar{\alpha}_T$ explains why gradient sparsification is communication efficient. In practice, only few top entries cover the majority of the gradient energy,  
  so $\alpha^i(T)$ grows rapidly for small values of $T$ and is much larger than $T/d$. 
%
  To illustrate the benefits of sparsification, let us look at the concrete example of logistic regression on the standard benchmark data set \texttt{RCV1} (with $d=47,236$ and $697,641$ data points).  
  Figure~\ref{fig:SGD_Alpha} depicts $\bar{\alpha}_T$ computed after running 1000 iterations of gradient descent and compares it to the worst case bound $T/d$. The results show a dramatic difference between these two measures. We quantify this difference by their ratio \vspace{-0.1cm}
  %
  \begin{align*}
  \texttt{SpeedUp}(T) &=  \frac{d}{T}\Bigg/ \frac{1}{\bar{\alpha}_T} =\frac{\bar \alpha_T}{T/d}. 
  \end{align*}
 Note that this measure is the ratio between rows 2 and 3 in
 Table~\ref{Tab:SDG_IC}, and hence, tells us the hypothetical speedup by sparsification, i.e., the ratio between the number of communicated floating points needed by GD and $T$-SGD to reach $\epsilon$-accuracy.
  The figure shows dramatic speedup; for small values of $T$, it is 3 order of magnitudes (we confirm this  in experiments below).
  
  Interestingly, the speedup  decreases with $T$ and is maximized at  $T=1$. This happens 
  because doubling $T$ doubles the amount of communicated bits, while the additional descent is often less significant.
 %
 %
  Thus, an increase in $T$ worsens communication efficiency. 
  This suggests that we should always take $T=1$ if the communication efficiency in terms of bits is optimized without considering overhead.
   In the context of the dynamic algorithm in Eq.~\eqref{eq:Alg1_G}, this leads to the following result:

%
%
%
  
  \begin{proposition}\label{prop:TisOneSimple}
     Consider the dynamic sparsified gradient algorithm in Eq.~\eqref{eq:Alg1_G} with $C(T)=P^{\texttt{S}}(T)$ given by Eq.~(\ref{eq:Payload_SQ}). 
     Then, the maximization problem~\eqref{eq:Alg1_G} has the solution $T^i=1$ for all $i$.
  \end{proposition}

 Figures~\ref{fig:SGD_TCC} and~\ref{fig:SGD_Experiments}  depict, respectively, the hypothetical and true values of the total number of bits needed to reach an $\epsilon$-accuracy for different communication models. In particular, Figure~\ref{fig:SGD_TCC} depicts the ratio $C(T)/\bar{\alpha}_T$ (compare with Table~\ref{Tab:SDG_IC})
 and Figure~\ref{fig:SGD_Experiments} depicts the experimental results of running $T$-SGD for different values of $T$.
 We consider: a) the payload model with $C(T)=P^{\texttt{S}}(T)$ (dashed lines) and b) the packet model in Eq.~\eqref{eq:COMM2} with $c_1=128$ bytes, $c_0=64$ bytes and $P_{\max}=128$ bytes (solid lines). In both cases, the floating point precision is $\texttt{FPP}=64$.  
 We compare the results with GD (blue lines) with payload $d\times \texttt{FPP}$ bits per iteration. 
 As expected, 
 if we ignore overheads then $T=1$ is optimal, and the improvement compared to GD are of the order of 3 magnitudes.  
For the packet model, there is a delicate balance between choosing $T$ too small or too big.
For general communication models it is difficult to find the right value of $T$
\emph{a priori},  and the costs of choosing a bad $T$ can be of many orders of magnitude.  To find a good $T$ we could do hyper-parameter search. Perhaps by first estimating $\bar{\alpha}_T$ from data and then use it to find optimal $T$. However, this will be expensive and, moreover,  $\bar{\alpha}_T$ might not be a good estimate of the $\alpha^i(T)$ we get at each iteration. 
 In contrast, our CAT framework  finds the optimal $T$ at each iteration without any hyper-parameter optimization. In Figure~\ref{fig:SGD_Experiments} we show the number of communicated bits needed to reach $\epsilon$-accuracy with our algorithm.  The results show that for both communication models, our algorithm achieves the same communication efficiency as if we would choose the optimal $T$.   

\section{Dynamic Sparsification + Quantization} \label{sec:DynSQ}


 We now describe how our CAT framework can improve the communication efficiency of compressed gradient methods that use sparsification combined with quantization, i.e., using $Q_T(\cdot)$ in Equation~\eqref{eq:T_sparse_quant}. 
%
 %
 %
 %
 %
 As before, our goal is to choose $T^i$  dynamically by maximizing the communication efficiency per iteration defined in (\ref{eq:Efficiency}). This selection can be performed based on the following descent lemma.
  \begin{lemma} 
  
   Suppose that $F:\R^d\rightarrow \R$ is (possibly non-convex) $L$-smooth.
   %
   Then for any $x,x^+\in\R^d$ with
   $ x^+=x-\gamma Q_T(\nabla F(x))$
  where $Q_T(\cdot)$ is as defined in Eq.~\eqref{eq:T_sparse_quant} and  $\gamma=\sqrt{\beta(T)}/(\sqrt{T} L)$ then
     $$ F(x^+)\leq  F(x) - \frac{\beta(T)}{2L} ||\nabla F(x)||^2, \qquad  \text{where} \qquad \beta(T)=  \frac{1}{T} \frac{\langle\nabla F(x),Q_T(\nabla F(x)) \rangle^2}{\Vert\nabla F(x)\Vert^4_2}.$$
   %
\label{lemma:Descent_SpQ}  
 \end{lemma}
 Since this compression operator affects the descent differently than sparsification, this lemma differs from 
 Lemma~\ref{lemma:Descent_sparse}, e.g,
%
in terms of the step-size and descent measure ($\beta(T)$ vs. $\alpha(T)$). 
  Unlike $\alpha(T)$ in Lemma~\ref{lemma:Descent_sparse}, $\beta(T)$ does not converge to 1 as $T$ goes to $d$. In fact, $\beta(T)$ is not even an increasing function, 
  and $Q_T(g)$ does not converge to $g$ when $T$ increases. 
  Nevertheless, $\langle \nabla F(x), Q_T(\nabla F(x)) \rangle^2$ is non-negative, increasing and concave. Under the affine communication model, $T\times C(T)=\tilde{c}_0 T^2 + c_1T$ is non-negative and convex, which implies that $\beta(T)/C(T)$ is quasi-concave. The optimal $T$ can then be efficiently found similarly to what was done for the CAT-sparsification in \S~\ref{subsec:benefitDynamicSparsifier}.
  %
 %
 %
 %
 Therefore,  Lemma~\ref{lemma:Descent_SpQ} allows us to apply the CAT framework for this compression. 
  In particular, with
 $\beta^i(T){=} (1/T) \langle\nabla F(x^i),Q_T(\nabla F(x^i)) \rangle^2/\Vert\nabla F(x^i)\Vert^4_2$
  %
    we get the algorithm
%
\begin{align}\label{eqn:CATS+Q_iteration}
   \textbf{Step 1: } T^i= \underset{T\in[1,d]}{\text{argmax}}~ \frac{\beta^i(T)}{ C(T) }, ~~\textbf{Step 2: } \gamma^i {=} \frac{\sqrt{\beta^i(T^i)}}{\sqrt{T^i} L},~~~
     \textbf{Step 3: } x^{i+1}{=}   x^i {-}\gamma^i  Q_{T^i}(\nabla F(x^i)). 
\end{align} 
 The algorithm optimizes $T^i$ depending on each gradient and the actual communication cost. Note that
  \citet{alistarh2017qsgd} propose a dynamic compression mechanism that chooses $T^i$
 so that $I_{T^i}(g^i)$ is the smallest subset such that  $\sum_{j\in I_{T^i}(g^i)} \vert g^i_j \vert \geq \| g^i \|_2$. However, this heuristic has no clear connection to descent or consideration for communication cost. Our experiments in \S~\ref{sec:Experiments} show that our framework outperforms this heuristic both  in terms of both running time and communication efficiency.  

\section{Dynamic Stochastic Sparsification: Stochastic Gradient \& Multiple Nodes} \label{sec:DynSS}


 We finally illustrate how the CAT framework can improve the communication efficiency of stochastic sparsification.
%
%
%
 Our goal is to choose $T^i$ and $p^i$  dynamically for the stochastic sparsification in Eq.~\eqref{eq:StochSpar} to maximize the communication efficiency per iteration. 
 To this end, we need the following descent lemma, similar to the ones we proved for 
deterministic sparsifications in the last two sections. 
  \begin{lemma}\label{lemma:Descent_SS}
  
   Suppose that $F:\R^d\rightarrow \R$ is (possibly non-convex) $L$-smooth.
   %
   Then for any $x,x^+\in\R^d$ with
   $ x^+=x-\gamma {Q}_{T,p}(\nabla F(x))$
   where ${Q}_{T,p}(\cdot)$ is defined in~\eqref{eq:StochSpar}
   and
    $\gamma=\omega_p(T)/ L$ we have
   \begin{align*}
       \mathbf{E} F(x^+)\leq  \mathbf{E} F(x) - \frac{\omega_p(T)}{2L} \mathbf{E}||\nabla F(x)||^2, \qquad \text{where} \quad \omega_p(T)=    \frac{||\nabla F(x)||^2}{ \mathbf{E} ||{Q}_{T,p}(\nabla F(x)) ||^2}. 
   \end{align*}
   %
 \end{lemma}
  Similarly as before, we optimize the descent and the communication efficiency by maximizing, respectively, $\omega_p(T)$ and $\omega_p(T)/C(T)$. 
  For given $T$, the   $p^{\star}$ minimizing $\omega_p(T)$  can be found efficiently, see~\cite{wang2018atomo} and our discussion in Appendix~\ref{app:DiscussionSM}. 
%
%
%
%
%
%
%
 %
%
%
%
 In this paper we always use $p^{\star}$ and omit $p$ in ${Q}_{T}(\cdot)$ and  $\omega(T)$.
 %
%
%
%
%
%
%
%
%
 We can now use our CAT framework to optimize the communication efficiency. If we set $\omega^i(T){=}||\nabla F(x^i)||^2/\mathbf{E} ||{Q}_{T,p}(\nabla F(x^i)) ||^2 $ we get the  dynamic algorithm:
\begin{align*}
   \textbf{Step 1: } T^i= \underset{B\in[1,d]}{\text{argmax}}~ \frac{\omega^i(T)}{ C(T) }, ~~~~\textbf{Step 2: } 
    \gamma^i = \frac{\omega^i(T^i)}{ L}, ~~~~\textbf{Step 3: }   x^{i+1}=   x^i -\gamma^i  {Q}_{T^i}(\nabla F(x^i)).
\end{align*} 
This algorithm can maximize communication efficiency by finding the optimal sparsity budget $T$ to the one-dimensional problem. This can be solved efficiently since the sparsification parameter $\omega(T)$ has properties that are similar to $\alpha(T)$ for deterministic sparsification. Like Lemma \ref{Lemma:LB_alpha} for deterministic sparsification, the following result shows that $\omega(T)$ is increasing with the budget $T\in[1,d]$ and is lower-bounded by $T/d$.      
\begin{lemma}\label{lemma:SS_LB_omega}
For $g\in\mathbb{R}^d$ the function 
\(    \omega(T) = {\|g\|^2}/{\| Q_{T,p}(g) \|^2}
\)
is increasing over $T\in [1,d]$. Moreover, $\omega(T)\geq T/d$ for all $T\in [1,d]$, where we obtain the equality when $p_j=T/d$ for all $j$.

\end{lemma}
This lemma leads to many consequences for $\omega(T)$, analogous to $\alpha(T)$. For instance, by following proof arguments in  Proposition \ref{prop:alpha_over_c}, $\omega(T)/C(T)$ attains its maximum for a $T$ which is an integer multiple of $\tau_{\max}$ when $C(T)= \tilde{c}_1 \lceil T / \tau_{\max} \rceil + c_0 $.

Furthermore, stochastic sparsification has some favorable properties that allow us to generalize our theoretical results to stochastic gradient methods and to multi-node settings. 
%
%
%
%
 Suppose that we have $n$ nodes that wish to solve the minimization problem with 
  \(F(x)= ({1}/{n}) \sum_{j=1}^n f_j(x)\)
  where $f_j(\cdot)$ is kept by node $j$\deleted{(in our setting, $f_j$ is the empirical loss of the data which resides in node $j$)}.
 Then, we may solve the problem by the distributed compressed gradient method 
\begin{align}\label{eqn:CompressedSGD}
    x^{i+1} = x^i - \gamma \frac{1}{n}\sum_{j=1}^n Q_{T_j^i}\left( g_j(x^i;\xi_j^i) \right),
\end{align}
where $Q(\cdot)$ is the stochastic sparsifier 
 and $g_j(x;\xi_j)$ is a stochastic gradient at $x$. 
 We assume that $g_j(x;\xi_j)$ is unbiased and satisfies 
a bounded variance assumption, i.e. $\mathbf{E}_{\xi}  g_j(x;\xi_j)  = \nabla f_j(x)$ and $\mathbf{E}_{\xi} \| g_j(x;\xi_j) - \nabla F(x) \|^2 \leq \sigma^2$.
The expectation is with respect to a local data distribution at node $j$. These conditions are standard to analyze first-order algorithms in machine learning \cite{feyzmahdavian2016asynchronous,lian2015asynchronous}.

We can derive a similar descent lemma 
for Algorithm \eqref{eqn:CompressedSGD}  as Lemma \ref{lemma:Descent_SS} for the single-node sparsification gradient method (see Appendix \ref{app:lemma:MultinodeSS_GD}). 
%
 This means that we easily prove similar data-dependent convergence results as we did for deterministic sparsification in Table~\ref{Tab:SDG_IC}.
To illustrate this, suppose that for a given $T$ there is $\bar{\omega}_T$  satisfying  $\omega_j^i(T) \geq \bar{\omega}_T$ where $\omega_j^i(T)= {||\nabla F_j(x^i)||^2}/{ \mathbf{E} ||{Q}_{T}(\nabla F_j(x^i)) ||^2}$. 
 Then the iteration complexity of Algorithm \eqref{eqn:CompressedSGD} is as given in the right part of  Table~\ref{Tab:SDG_IC}. 
 The parameter $\bar{\omega}_T$ captures the sparsification gain, similarly as $\bar{\alpha}_T$ did for deterministic sparsification. In the worst case there is no communication improvement of sparsification compared to sending full gradients, but when $\bar{\omega}_T$ is large the communicaton improvment can be significant.

\section{Experimental Results} \label{sec:Experiments}

\textbf{Experiment 1 (single node).} %
We evaluate the performance of our CAT framework for dynamic sparsification and
quantization (S+Q) in the single-master, single-worker setup on the \texttt{URL}
data set with  $2.4$ million data points and $3.2$ million features. The master node, located 500 km away from the worker node, is
responsible for maintaining the decision variables based on the gradient
information received from the worker node. 
The nodes communicate with each other over a 1000
Mbit Internet connection using the \texttt{ZMQ} library. 
We implemented
vanilla gradient descent (GD), Alistarh's S+Q~\cite{alistarh2017qsgd} and CAT
S+Q using the C++ library \texttt{POLO}~\cite{aytekin2018polo}. We first set
$\texttt{FPP} = 32$ and measure the communication cost
in wall-clock time. After obtaining a linear fit to the measured communication cost
(see supplementary for details), we ran $30,000$ iterations and with step-size according to Lemma~\ref{lemma:Descent_SpQ}.
Figure~\ref{fig:MPI_multiplenodes} shows the loss improvement with respect to
the total communication cost (leftmost) and wall-clock time (middle). We observe that CAT S+Q outperforms GD and Alistarh's S+Q up to
two orders and one order of magnitude, respectively, in communication
efficiency. In terms of wall-clock time, CAT S+Q takes 26\% (respectively, 39\%)
more time to finish the full $30,000$ iterations than that of GD (respectively,
Alistarh's S+Q). Note, however, that CAT S+Q achieves an order of magnitude loss
improvement in an order of magnitude shorter time, and the loss value is always
lower in CAT S+Q than that in Alistarh's S+Q. Such a performance is desirable in
most of the applications (e.g., hyper-parameter optimizations and day-ahead
market-price predictions) that do not impose a strict upper bound on the number
of iterations but rather on the wall-clock runtime of the algorithm.

\textbf{Experiment 2 (MPI - multiple nodes):} 
We evaluate the performance of our CAT tuning rules on  deterministic sparsification (SG), stochastic sparsification (SS), and sparsification with quantization (S+Q)  in a multi-node setting on \texttt{RCV1}.
We compare the results to gradient descent and Alistarh's S+Q~\cite{alistarh2017qsgd}.
 We implement all algorithms in Julia, and run them on $4$ nodes using MPI, splitting the data evenly between the nodes.
%
%
%
In all cases we use the packet communication model~\eqref{eq:COMM2} with $c_1=576$ bytes, $c_0=64$ bytes and $P_{\max}=512$ bytes. 
The rightmost plot in Figure~\ref{fig:MPI_multiplenodes} shows that our CAT S+Q outperforms all other compression schemes.
 In particular, CAT is roughly 6 times more communication efficient than  the dynamic rule in \cite{alistarh2017qsgd} for the same compression scheme (compare number of bits needed to reach $\epsilon=0.4$).  

 \begin{figure}
    \centering
  \includegraphics[width = \textwidth]{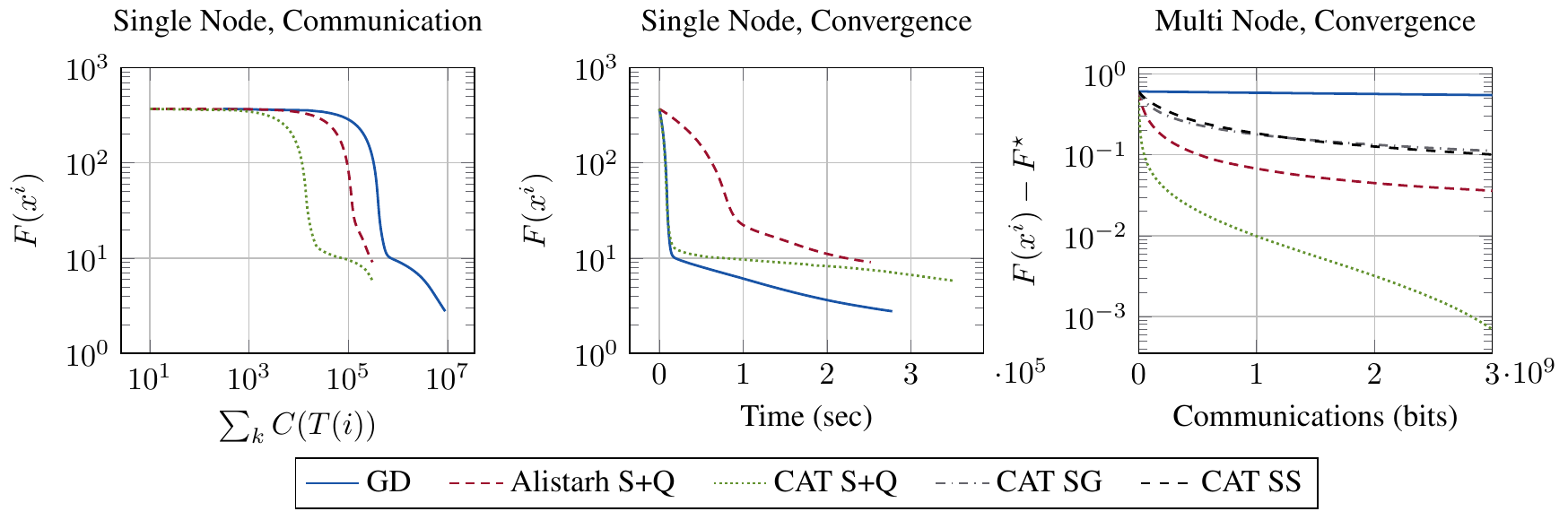}
    \caption{ Performance of gradient methods using CAT frameworks on three main compressors for solving logistic regression problems. We used the \texttt{URL} data set in the single-node (one master/one worker) architecture and \texttt{RCV1} in the multi-node (one master/four worker) setting.
    }
    \label{fig:MPI_multiplenodes}
\end{figure}

 %
\section{Conclusions}

We have proposed communication-aware adaptive tuning to optimize the communication-efficiency of gradient sparsification.
%
%
The adaptive tuning relies on a data-dependent measure of objective function improvement, and adapts the compression level to maximize the descent per communicated bit. 
%
%
Unlike existing heuristics, our tuning rules are guaranteed to save communications in realistic communication models. 
In particular, our rules is more communication-efficient when communication overhead or packet transmissions are accounted for. 
%
%
In addition to centralized analysis, our tuning strategies are proven to reduce communicated bits also in distributed scenarios.  

\bibliography{refs_SK_modified}

\newpage
\section{Appendix}
\appendix
\section{Proofs of Lemmas and Propositions}

\subsection{Proof of Lemma \ref{lemma:Descent_sparse} } 
By the $L$-smoothness of $F(\cdot)$  and the iterate $x^{+}=x-\gamma Q_T(\nabla F(x))$ where $x^+, x\in\mathbb{R}^d$, from  Lemma 1.2.3. of \cite{nesterov2018lectures} we have
\begin{align*}
    F(x^+) 
    &\leq F(x) - \gamma \left\langle \nabla F(x) , Q_T(\nabla F(x)) \right\rangle \\
    &\hspace{0.4cm}+\frac{L\gamma^2}{2} \| Q_T(\nabla F(x)) \|^2.
\end{align*}
It can be verified that $$\left\langle g, Q_T(g) \right\rangle= ||Q_T(g)||^2$$
for all $g\in \R^d$ 
and, therefore, if $\gamma = 1/L$ then we have 
\begin{align*}
    F(x^+) \leq F(x) - \frac{1}{2L}\| Q_T(\nabla F(x)) \|^2.
\end{align*}
By the definition of $\alpha(T)$ we have $\| Q_T(\nabla F(x)) \|^2 = \alpha(T)\| \nabla F(x) \|^2$, which yields the result.

Next, we prove that there exist $L$-smooth functions $F(\cdot)$ where the inequality is tight. Consider $F(x) = L\|x\|^2/2$. Then, $F$ is $L$-smooth, and also satisfies  
\begin{align*}
    &F( x - \gamma Q_T(\nabla F(x)) )  \\
    &\hspace{0.5cm} = \frac{L}{2} \| x - \gamma Q_T(Lx) \|^2 \\
    &\hspace{0.5cm} = \frac{L}{2}\| x \|^2 - \gamma \langle  Lx , Q_T(Lx)\rangle + \frac{L\gamma^2}{2}\| Q_T(Lx)\|^2.
\end{align*}
Since  $\langle g , Q_T(g)\rangle  = \| Q_T(g) \|^2$ by the definition $Q_T(\cdot)$ and $\gamma = 1/L$, we have   
\begin{align*}
    F( x - \gamma Q_T(\nabla F(x)) ) 
    &  = F(x)  -  \frac{1}{2L}\| Q_T(Lx)\|^2.
\end{align*}
Since $\nabla F(x)=Lx$, by the definition of $\alpha(T)$ 
\begin{align*}
\alpha(T) &= \frac{\| Q_T(L x) \|^2}{ \| L x\|^2}  = \frac{\sum_{i\in I_T} x_i^2}{ \sum_{i=1}^d x_i^2},
\end{align*}
where $I_T$ is the index set of $T$ elements with the highest absolute magnitude. Therefore,  
\begin{align*}
    F( x - \gamma Q_T(\nabla F(x)) ) 
    &  = F(x)  -  \frac{\alpha(T)}{2L}\| \nabla F(x)\|^2.
\end{align*}

\subsection{Proof of Lemma \ref{Lemma:LB_alpha}}

Take $g\in\mathbb{R}^d$ and, 
without the loss of generality, we let 
$\vert g_1\vert \geq \vert g_2 \vert \geq \ldots \geq  \vert g_d \vert$ and $g_i\in\mathbb{R}$ (otherwise we may re-order $g$). 
To prove that $\alpha(T)$ is increasing we rewrite the definition of $\alpha(T)$ equivalently as
\begin{align*}
    \alpha(T) = \sum_{j=1}^T g_j^2/\|g\|^2, \quad \text{for} \quad T \in \{0,1,2,\ldots,d\}.
\end{align*}
Notice that $\alpha(T)=0$ when $T=0$. Clearly, $\alpha(T)$ is also increasing with $T\in[1,d]$ since each term of the sum $\sum_{j=1}^T g_j^2$ is increasing.
%
%

We prove that $\alpha(T)$ is concave by recalling the slope of $\alpha(T)$
\begin{align*}
    \frac{d}{dT}\alpha(T) = g_M^2/\|g\|^2, 
\end{align*}
for $T \in (M-1,M)$ and $M=1,2,\ldots,d$. 
Since $|g_1|\geq |g_2| \geq \ldots \geq |g_d|$, the slope of $\alpha(T)$ has a non-increasing slope when $T$ increases. Therefore, $\alpha(T)$ is concave.

We prove the second statement by writing $\|g\|^2$ on the form of
$$
\| g \|^2  = \sum_{j\in I_T(g)} g_j^2 + \sum_{j \in I_{T^c}(g)} g_j^2,
$$
where $I_{T^c}$ is the index set of $d-T$ elements with lowest absolute magnitude.  
Applying the fact that $g_j^2 \leq  \min_{l \in I_T(g)} g_l^2$ for $j\in I_{T^c}(g)$ and that $\min_{l \in I_{T}(g)} g_l^2 \leq (1/T) \sum_{l \in I_T(g)} g_l^2$ into the main inequality, we have 
\begin{align*}
    \| g\|^2 & \leq \left( 1+\frac{d-T}{T} \right) \sum_{j\in I_T(g)} g_j^2.
\end{align*}
By the definition of $Q_T(g)$, we get
$$\alpha(T)\geq T/d.$$

Finally, 
we prove the last statement by  setting $F(x) = (1/2)x^TAx$ where $A=(L/d) \mathbf{1}\mathbf{1}^T$.
 Then $F(\cdot)$ is $L$-smooth  and its gradient is 
$$\nabla F(x) =  \bar{x} \mathbf{1},$$
where $$\bar{x}=\frac{1}{d}\sum_{i=1}^d x_i.$$
Therefore, $\| Q_T(\nabla F(x)) \|^2 = (T/d)\| \nabla F(x) \|^2.$

\subsection{Proof of Proposition \ref{prop:alpha_over_c}}


The ratio between a non-negative concave function $\alpha(T)$ and a positive affine function $C(T)$ is quasi-concave and semi-strictly quasi-concave~\cite{Schaible2013,Avriel2010}, meaning that every local maximal point is globally maximal.
%

Next, we consider $\alpha(T)/C(T)$ when $C(T)= \tilde{c}_1 \lceil T / \tau_{\max} \rceil + c_0.$ If $T \in ( (c-1)\tau_{\max} , c \tau_{\max}] $, then $\alpha(T) \leq \alpha(c\tau_{\max})$ due to monotonicity of $\alpha(\cdot)$ and $C(T) = C(c\tau_{\max})$, meaning that  $\alpha(T)/C(T) \leq \alpha(c\tau_{\max})/C(c\tau_{\max})$. This implies that $T=c\tau_{\max}$ maximizes $\alpha(T)/C(T)$ for $T \in ( (c-1)\tau_{\max} , c \tau_{\max}] $ and that we can obtain the maximum of $T = c\tau_{\max}$ for some integers $c$.

\subsection{Proof of Proposition \ref{prop:TisOneSimple}}
Take $g\in\mathbb{R}^d$ and, 
without the loss of generality, we let 
$\vert g_1\vert \geq \vert g_2 \vert \geq \ldots \geq  \vert g_d \vert$ and $g_i\in\mathbb{R}$ (otherwise we may re-order $g$). 
Since $C(T) = C \cdot T$ where $C=\lceil \log_2(d) \rceil + \texttt{FPP}$, we have
\begin{align*}
    T^i = \mathop\text{argmax}\limits_{T\in[1,d]} \frac{\alpha^i(T)}{C(T)} =  \mathop\text{argmax}\limits_{T\in[1,d]} \frac{\sum_{i=1}^T g_i^2}{C\cdot T}. 
\end{align*}
Since ${\sum_{i=1}^T g_i^2}/{T} \leq g_1^2$, the solution from Equation \eqref{eq:Alg1_G} is $T^i=1$ for all $i$.

\subsection{Proof of Lemma \ref{lemma:Descent_SpQ}   }
By using the $L$-smoothness of $F(\cdot)$ (Lemma 1.2.3. of \cite{nesterov2018lectures}) and the iterate $x^+=x-\gamma Q_T(\nabla F(x))$ where $x^+, x\in\mathbb{R}^d$, we have 
\begin{align*}
    F(x^+) 
    &\leq F(x) - \gamma \left\langle \nabla F(x) ,Q_T(\nabla F(x)) \right\rangle \\
    &\hspace{0.4cm} +\frac{L\gamma^2}{2} \| Q_T(\nabla F(x)) \|^2.
\end{align*}
If $Q_T(\nabla F(x))$ has $T$ non-zero elements, then we can easily prove that 
\begin{align*}
    \left\langle \nabla F(x) ,Q_T(\nabla F(x)) \right\rangle 
    & =  \sqrt{T\beta(T)} \cdot \| \nabla F(x) \|^2, \quad \text{and} \\
    \| Q_T(\nabla F(x)) \|^2 & = T\cdot\| \nabla F(x) \|^2,
\end{align*}
where $\beta(T)$ is defined as
\begin{align*}
    \beta(T)=  \frac{1}{T} \frac{\langle\nabla F(x),Q_T(\nabla F(x)) \rangle^2}{\Vert\nabla F(x)\Vert^4_2}.
\end{align*}
Plugging these equations into the above inequality yields  
\begin{align*}
    F(x^+) \leq F(x) -\left( \gamma  \sqrt{T\beta(T)} - \frac{T L\gamma^2}{2} \right) \| \nabla F(x)\|^2. 
\end{align*}
Setting $\gamma = \sqrt{\beta(T)}/(\sqrt{T} L)$ completes the proof.

\subsection{Proof of Lemma \ref{lemma:Descent_SS} } 
By using the $L$-smoothness of $F(\cdot)$ (Lemma 1.2.3. of \cite{nesterov2018lectures}) and the iterate $x^+=x-\gamma Q_{T,p}(\nabla F(x))$ where $x^+, x\in\mathbb{R}^d$, we have 
\begin{align*}
    F(x^+) & \leq F(x) - \gamma \left\langle \nabla F(x) ,Q_{T,p}(\nabla F(x)) \right\rangle \\
    &\hspace{0.4cm}+\frac{L\gamma^2}{2} \| Q_{T,p}(\nabla F(x)) \|^2.
\end{align*}
Since $\omega_p(T)$ is defined by 
\begin{align*}
    \omega_p(T)=    \frac{||\nabla F(x)||^2}{ \mathbf{E} ||{Q}_{T,p}(\nabla F(x)) ||^2},
\end{align*}
taking the expectation, and using the unbiased property of $Q_{T,p}(\cdot)$ we get
\begin{align*}
    \mathbf{E} F(x^+) \leq \mathbf{E} F(x) - \left(\gamma - \frac{L\gamma^2}{2\omega_p(T)} \right) \mathbf{E} \| \nabla F(x) \|^2.
\end{align*}
Now taking $\gamma = \omega_p(T)/L$ concludes the complete the proof.

\subsection{Proof of Lemma \ref{lemma:SS_LB_omega} }
Consider $Q_{T,p}(\cdot)$ in Equation \eqref{eq:StochSpar}. Then, 
\begin{align*}
    \mathbf{E} \| Q_{T,p} (g) \|^2 = \sum_{j=1}^d \frac{1}{p_j} g_j^2.
\end{align*}
Here, we assume without the loss of generality that each element of 
$g\in\mathbb{R}^d$ is $g_i$ such that $\vert g_1\vert \geq \vert g_2\vert \geq \ldots\geq \vert g_d \vert$ (otherwise we may re-order $g$). Therefore, 
\begin{align*}
    \omega(T) = \frac{\sum_{j=1}^d g_j^2}{\sum_{j=1}^d \frac{1}{p_j} g_j^2}. 
\end{align*}
To ensure the high sparsity budget $T$ of the compressed gradient $Q_{T,p}(g)$, probabilities must also have high values in some coordinates (some $p_j$ are close to one). Therefore, $\omega(T)$ is increasing over the sparsity budget $T\in[1,d]$.

Next, we assume that $Q_{T,p}(\cdot)$ in Equation \eqref{eq:StochSpar}  has $p_j=T/d$ for all $j$. Then, 
\begin{align*}
    \mathbf{E} \| Q_{T,p} (g) \|^2 = \frac{d}{T}\|g\|^2.
\end{align*}
Plugging this result into the main definition, we have $\omega(T)=T/d$.
Since we assign $p$ that minimizes $\omega_p(T) = \|g \|^2/\mathbf{E}\|g\|^2$,  $\omega_p(T)\geq T/d$.



\section{Iteration Complexities of Adaptive Compressors}\label{app:Table1}

In this section, we provide the iteration complexities of gradient descent \eqref{Alg:main} with three main compressors:  
deterministic sparsification \eqref{eq:T_sparse_operator}, dynamic sparsification together with quantization \eqref{eq:T_sparse_quant} , and stochastic sparsification \eqref{eq:StochSpar}.

\subsection{Analysis for Deterministic Sparsification}

We provide theoretical convergence guarantees for gradient descent using deterministic sparsification. 

\begin{theorem}\label{thm:TopKSG_thm}
Consider the minimization problem over the function $F(x)$ and the iterates $\{x^i\}_{i\in\mathbb{N}}$ generated by gradient descent with dynamic sparsification in Equation \eqref{eq:Alg1_G}. Suppose that there exists $\bar\alpha_T\in[0,1]$ such that ${\alpha}^i(T) \geq \bar\alpha_T \geq T/d$ for all $i$. Set $\epsilon_0 = F(x^0) - F(x^\star)$. Then, 
\begin{enumerate}[leftmargin=0.3in]
    \item \textbf{Non-convex:} If $F$ is $L$-smooth, then we find ${\min}_{l\in[0,i-1]}\| \nabla F(x^l) \|^2 \leq\epsilon$ in 
        \begin{align*}
        i =  \frac{1}{\bar \alpha_T} \frac{2L\epsilon_0}{\epsilon} \quad \text{iterations.}
    \end{align*}
    \item \textbf{Convex:} If $F$ is also convex and there exists a positive constant $R$ such that $\| x^i-x^\star \|\leq R$, then  we find $F(x^i) - F(x^\star)\leq \epsilon$ in 
    \begin{align*}
        i = \frac{1}{\bar \alpha_T} \frac{2L R^2}{\epsilon}\quad \text{iterations.}
    \end{align*}
    \item \textbf{Strongly-convex:} If $F$ is also $\mu$-strongly convex, then we find $F(x^i) - F(x^\star) \leq \epsilon$ in 
        \begin{align*}
        i = \frac{1}{\bar\alpha_T}\kappa \log\left(\frac{\epsilon_0}{\epsilon}\right)\quad \text{iterations.}
    \end{align*}
\end{enumerate}

\end{theorem}
\begin{proof}
See Appendix \ref{app:thm:TopKSG_thm}. 
\end{proof}

\subsection{Analysis for Dynamic Sparsification together with Quantization}

We prove convergence rate of gradient descent with dynamic sparsification together with quantization. 

\begin{theorem}\label{thm:S+Q_GD_thm}
Consider the minimization problem over the function $F(x)$ and the iterates $\{x^i\}_{i\in\mathbb{N}}$ generated by gradient descent with sparsification together with quantization  in Equation \eqref{eqn:CATS+Q_iteration}. Suppose that there exist constants $T\in (0,d]$ and $\bar\beta_T\in(0,\infty)$  such that
$ T^i \leq T \text{ and } {\beta}^i(T^i) \geq \bar\beta_T$ for all $i$, respectively. Set $\epsilon_0 = F(x^0) - F(x^\star)$. Then,  
\begin{enumerate}[leftmargin=0.3in]
     \item \textbf{Non-convex:} If $F$ is $L$-smooth, then we find ${\min}_{l\in[0,i-1]}\| \nabla F(x^l) \|^2 \leq\epsilon$ in 
        \begin{align*}
        i =  \frac{1}{\bar{\beta}_T} \frac{2L\epsilon_0}{\epsilon} \quad \text{iterations.}
    \end{align*}
    \item \textbf{Convex:} If $F$ is also convex and there exists a positive constant $R$ such that $\| x^i-x^\star \|\leq R$, then  we find $F(x^i) - F(x^\star)\leq \epsilon$ in 
    \begin{align*}
        i = \frac{1}{\bar{\beta}_T} \frac{2L R^2}{\epsilon}\quad \text{iterations.}
    \end{align*}
    \item \textbf{Strongly-convex:} If $F$ is also $\mu$-strongly convex, then we find $F(x^i) - F(x^\star) \leq \epsilon$ in 
        \begin{align*}
        i = \frac{1}{\bar{\beta}_T}\kappa \log\left(\frac{\epsilon_0}{\epsilon}\right)\quad \text{iterations.}
    \end{align*}
\end{enumerate}

\end{theorem}
\begin{proof}
See Appendix \ref{app:thm:S+Q_GD_thm}. 
\end{proof}

\subsection{Analysis for Stochastic Sparsification}

We prove iteration complexities of stochastic gradient descent with stochastic sparsification in the multi-node setting. 

\begin{theorem}\label{thm:mutlinode_SS_SGD}
Consider the minimization problem over the function $F(x)=(1/n)\sum_{l=1}^n f_l(x)$, where each $f_l(\cdot)$ is $L$-smooth. Let the iterates  $\{x^i\}_{i\in\mathbb{N}}$ generated by Algorithm~\eqref{eqn:CompressedSGD}, and  suppose that there exists $\bar{\omega}_T$ such that  $\omega_j^i(T_j^i)\geq \bar{\omega}_T$ where $\omega_j^i$ is the sparsification parameter of node $j$ at iteration $i$. Set $\epsilon^F_0 = F(x^0)-F(x^\star)$ and $\epsilon^X_0=\| x^0 - x^\star \|$.  Then, 

\begin{enumerate}[leftmargin=0.3in]
    \item \textbf{Non-convex} If 
    $$\gamma = \frac{\bar{\omega}_T}{2L}  \frac{1}{2\sigma^2/\epsilon+1}$$
    then we find $\min_{l\in[0,i-1]} \mathbf{E}\| \nabla F(x^l) \|^2\leq\epsilon$ in 
    \begin{align*}
        i & = \frac{2}{\bar{\omega}_T} \left( 1 + \frac{2\sigma^2}{\epsilon} \right)\cdot\frac{2L\epsilon^F_0}{\epsilon}  \quad \text{iterations}.
    \end{align*}
    \item \textbf{Convex} If $F$ is convex and
    $$\gamma = \frac{\bar{\omega}_T}{2} \frac{1}{2\sigma^2/\epsilon+L} $$
    then we find $\mathbf{E} \left[ F(\sum_{l=0}^{i-1} x^l/i) - F^\star \right] \leq \epsilon$
     \begin{align*}
        i & = \frac{2}{\bar{\omega}_T} \left( 1+\frac{2\sigma^2}{\epsilon L} \right) \frac{2L\epsilon^X_0}{\epsilon}  \quad \text{iterations}.
    \end{align*}
    \item \textbf{Strongly-convex} If $F$ is $\mu$-strongly convex and 
    $$ \gamma = \frac{\bar{\omega}_T}{2} \frac{1}{2\sigma^2/(\mu\epsilon)+L}$$
    then we find $\mathbf{E}[F(x^i)- F(x^\star)] \leq \epsilon$ in 
    \begin{align*}
        i & = \frac{2}{\bar{\omega}_T} \kappa \left( 1+ \frac{2\sigma^2}{\mu\epsilon L}  \right) \log\left(  \frac{2\epsilon^F_0}{\epsilon}\right)  \quad \text{iterations}.
    \end{align*}
\end{enumerate}

\end{theorem}
\begin{proof}
See Appendix \ref{app:thm:mutlinode_SS_SGD}.
\end{proof}

\section{Proof of Theorem \ref{thm:TopKSG_thm}} \label{app:thm:TopKSG_thm}

In this section, we derive iteration complexities of gradient descent with deterministic sparsification.

\subsection*{Proof of Theorem \ref{thm:TopKSG_thm}-1}

By recursively applying the inequality from Lemma \ref{lemma:Descent_sparse} with $x^+ = x^{i+1}$ and $x=x^i$, we have 
\begin{align*}
    \min_{l\in[0,i-1]} \| \nabla F(x^l) \|^2 \leq \frac{2L}{i} \sum_{l=0}^{i-1} \frac{1}{\alpha^i(T)} [ F(x^l) - F(x^{l+1}) ]
\end{align*}
where the inequality follows from the fact that  $\min_{l\in[0,i-1]} \| \nabla F(x^l) \|^2 \leq \sum_{l=0}^{i-1} \| \nabla F(x^l) \|^2/i$. If there exists $\bar \alpha_T$ such that $\alpha^i(T) \geq \bar\alpha_T$, then 
\begin{align*}
    \min_{l\in[0,i-1]} \| \nabla F(x^l) \|^2 \leq \frac{2L}{i \bar\alpha_T} [F(x^0) - F(x^{i})]. 
\end{align*}
Since $F(x)\geq F(x^\star)$ for $x\in\mathbb{R}^d$, 
\begin{align*}
    \min_{l\in[0,i-1]} \| \nabla F(x^l) \|^2 \leq \frac{2L}{i \bar\alpha_T} \epsilon_0, 
\end{align*}
where $\epsilon_0 = F(x^0)-F(x^\star)$. This means that to reach $\epsilon$-accuracy (i.e. $\min_{l\in[0,i-1]} \| \nabla F(x^l) \|^2 \leq \epsilon$), the sparsified gradient method \eqref{Alg:main} needs at most 
%
%
\begin{align*}
     i \geq \frac{1}{\bar\alpha_T}\frac{2L\epsilon_0}{\epsilon} ~~\text{iterations}.
\end{align*}

In addition, we recover the iteration complexities of the sparsified gradient method and of classical full gradient method when we let $\bar \alpha_T = T/d$ and $\bar \alpha_T = 1$, respectively.

\subsection*{Proof of Theorem \ref{thm:TopKSG_thm}-2}

Before deriving the result, we introduce one useful lemma. 
\begin{lemma}\label{lemma:TrickConvex}
The non-negative sequence $\{V^i\}_{k\in\mathbb{N}}$ generated by 
\begin{align}\label{eqn:TrickConvex}
V^{i+1} \leq V^i - q (V^i)^2, \quad \text{for} \quad q>0
\end{align}
satisfies 
\begin{align}
\frac{1}{V^i} \geq \frac{1}{V^0} + i q.
\end{align}
\end{lemma}
\begin{proof}
By the fact that $x^2\geq 0$ for $x\in\mathbb{R}$, clearly $V^{i+1}\leq V^i$. By the proper manipulation, we rearrange the terms in Equation \eqref{eqn:TrickConvex} as follows:
\begin{align*}
\frac{1}{V^{i+1}} - \frac{1}{V^i} \geq q \frac{V^i}{V^{i+1}} \geq q,
\end{align*}
where the last inequality follows from the fact that $V^{i+1}\leq V^i$. By the recursion, we complete the proof.
\end{proof}

By Lemma \ref{lemma:Descent_sparse} with $x^+ = x^{i+1}$ and $x=x^i$, we have 
\begin{align*}
    F(x^{i+1}) \leq F(x^i) - \frac{\alpha^i(T)}{2L} \| \nabla F(x^i) \|^2
\end{align*}

Since $F$ is convex, i.e. 
\begin{align*}
 F(x) -F(x^\star) \leq \langle \nabla F(x), x - x^\star \rangle, \ \text{for} \ x\in\mathbb{R}^d
\end{align*}
by Cauchy-Scwartz's inequality and assuming that the iteration satisfies $\| x - x^\star \| \leq R$ for $R>0$ and $x\in\mathbb{R}^d$, 
\begin{align*}
  \|\nabla F(x)\| \geq \frac{1}{R}\left[ F(x) -F(x^\star) \right].
\end{align*}
Plugging this inequality into the main result, we have 
\begin{align*}
    V^{i+1} \leq V^i - \frac{\alpha^i(T)}{2L R^2} (V^i)^2, 
\end{align*}
where $V^i = F(x^i) - F(x^\star).$ If there exists $\bar \alpha_T$ such that $\alpha^i(T) \geq \bar\alpha_T$, then by Lemma \ref{lemma:TrickConvex} and by using the fact that $V^0\geq 0$ 
\begin{align*}
V^i \leq  \frac{1}{\bar\alpha_T}\frac{2LR^2}{i}.
\end{align*}
To reach $F(x^i) - F(x^\star) \leq\epsilon$, the sparsified gradient methods needs the number of iterations $i$ satisfying 
\begin{align*}
    i \geq \frac{1}{\bar\alpha_T}\frac{2LR^2}{\epsilon}.
\end{align*}
We also recover the iteration complexities of the sparsified gradient method and of classical full gradient method when we let $\bar \alpha_T = T/d$ and $\bar \alpha_T = 1$, respectively.

\subsection*{Proof of Theorem \ref{thm:TopKSG_thm}-3}
By Lemma \ref{lemma:Descent_sparse} with $x^+ = x^{i+1}$ and $x=x^i$, we have 
\begin{align*}
    F(x^i) \leq F(x^i) - \frac{\alpha^i(T)}{2L} \| \nabla F(x^i) \|^2.
\end{align*}
Since $F$ is $\mu$-strongly convex, i.e. $\| \nabla F(x) \|^2 \geq 2\mu [F(x) - F(x^\star)]$ for $x\in\mathbb{R}^d$, applying this inequality into the main result we have
\begin{align*}
     F(x^{i+1}) - F(x^\star) \leq  \left( 1 - \frac{\mu \alpha^i(T)}{L} \right) [F(x^i) - F(x^\star)].
\end{align*}
If there exists $\bar \alpha_T$ such that $\alpha^i(T) \geq \bar\alpha_T$, then by the recursion we get
\begin{align*}
     F(x^{i}) - F(x^\star) \leq  \left( 1 - \frac{\mu \alpha_T}{L} \right) ^i \epsilon_0,
\end{align*}
where $\epsilon_0= F(x^0) - F(x^\star)$. To reach $F(x^i) - F(x^\star) \leq\epsilon$, the sparsified gradient methods requires the number of iterations $i$ satisfying 
\begin{align*}
    \left( 1 - \frac{\mu \alpha_T}{L} \right) ^i \epsilon_0 \leq \epsilon.
\end{align*}
Taking the logarithm  on both sides of the inequality and using the fact that $-1/\log(1-x)\leq 1/x$ for $0<x\leq 1$, we have 
\begin{align*}
    i \geq  \frac{1}{\bar\alpha_T}\kappa \log\left(  \frac{\epsilon_0}{\epsilon} \right).
\end{align*}
We also recover the iteration complexities of the sparsified gradient method and of classical full gradient method when we let $\bar \alpha_T = T/d$ and $\bar \alpha_T = 1$, respectively.

\section{Proof of Theorem \ref{thm:S+Q_GD_thm}} \label{app:thm:S+Q_GD_thm}
In this section, we prove the iteration complexities of gradient descent using dynamic sparsification together with quantization. 

\subsection*{Proof of Theorem \ref{thm:S+Q_GD_thm}-1}

Suppose that there exist $T$ and $\bar{\beta}_T$  such that $T^i \leq T$ and $\beta^i(T^i)\geq \bar {\beta}_T$ for all $i$, respectively. Applying  Lemma \ref{lemma:Descent_SpQ} with $x^+ = x^{i+1}$ and $x=x^i$ we then have 
\begin{align*}
   \min_{l\in[0,i-1]} \| \nabla F(x^l) \|^2 \leq  \frac{2L}{\bar{\beta}_T i }\sum_{l=0}^{i-1} [F(x^l) - F(x^{l+1})],
\end{align*}
where the inequality results from the fact that  $\min_{l\in[0,i-1]} \| \nabla F(x^l) \|^2 \leq \sum_{l=0}^{i-1} \| \nabla F(x^l) \|^2/i$. By the cancellations of the telescopic series, and by the fact that $F(x)\geq F(x^\star)$ for $x\in\mathbb{R}^d$, 
\begin{align*}
    \min_{l\in[0,i-1]} \| \nabla F(x^l) \|^2 \leq \frac{2L}{\bar{\beta}_T i} \epsilon_0,
\end{align*}
where $\epsilon_0 = F(x^0)-F(x^\star)$. To reach $\min_{l\in[0,i-1]} \| \nabla F(x^l) \|^2\leq \epsilon$, gradient descent using sparsification and quantization needs at most
\begin{align*}
    i \geq \frac{2L}{\bar{\beta}_T} \frac{\epsilon_0}{\epsilon} \quad \text{iterations}. 
\end{align*}

\subsection*{Proof of Theorem \ref{thm:S+Q_GD_thm}-2}
Since $\beta^i(T^i) \geq \bar {\beta}_T$ for all $i$, applying  Lemma \ref{lemma:Descent_SpQ} with $x^+ = x^{i+1}$ and $x=x^i$ we have 
\begin{align*}
    [F(x^{i+1}) - F(x^\star)] \leq [F(x^i) - F(x^\star)] - \frac{\bar{\beta}_T}{2L}\|  \nabla F(x^i)\|^2.
\end{align*}
Since $F$ is convex, i.e. 
\begin{align*}
 F(x) -F(x^\star) \leq \langle \nabla F(x), x - x^\star \rangle, \ \text{for} \ x\in\mathbb{R}^d
\end{align*}
by Cauchy-Scwartz's inequality and assuming that the iteration satisfies $\| x - x^\star \| \leq R$ for $R>0$ and $x\in\mathbb{R}^d$, 
\begin{align*}
  \|\nabla F(x)\| \geq \frac{1}{R}\left[ F(x) -F(x^\star) \right].
\end{align*}
Plugging this inequality into the main result, we have 
\begin{align*}
V^i \leq  V^i - \frac{\bar{\beta}_T}{2LR^2}(V^i)^2,
\end{align*}
where $V^i = F(x^i)-F(x^\star)$. Applying  Lemma \ref{lemma:TrickConvex} with $V^0\geq 0$ into this inequality we get 
\begin{align*}
    V^i \leq \frac{1}{\bar{\beta}_T}\frac{2LR^2}{i}. 
\end{align*}

To reach $F(x^i) - F(x^\star) \leq\epsilon$, gradient descent using sparsification with quantization needs the number of iterations $i$ satisfying 
\begin{align*}
    i \geq \frac{1}{\bar{\beta}_T}\frac{2LR^2}{\epsilon}.
\end{align*}

\subsection*{Proof of Theorem \ref{thm:S+Q_GD_thm}-3}
Since $\beta^i(T^i) \geq \bar {\beta}_T$ for all $i$, applying  Lemma \ref{lemma:Descent_SpQ} with $x^+ = x^{i+1}$ and $x=x^i$ we have
%
\begin{align*}
    F(x^i) \leq F(x^i) - \frac{\bar{\beta}_T}{2L} \| \nabla F(x^i) \|^2.
\end{align*}
Since $F$ is $\mu$-strongly convex, i.e. $\| \nabla F(x) \|^2 \geq 2\mu [F(x) - F(x^\star)]$ for $x\in\mathbb{R}^d$, applying this inequality into the main result we have
\begin{align*}
     F(x^{i+1}) - F(x^\star) \leq  \left( 1 - \frac{\mu \bar{\beta}_T}{L} \right) [F(x^i) - F(x^\star)].
\end{align*}
By the recursion, we get
\begin{align*}
     F(x^{i}) - F(x^\star) \leq  \left( 1 - \frac{\mu \bar{\beta}_T}{L} \right) ^i \epsilon_0,
\end{align*}
where $\epsilon_0= F(x^0) - F(x^\star)$. To reach $F(x^i) - F(x^\star) \leq\epsilon$, the sparsified gradient methods requires the number of iterations $i$ satisfying 
\begin{align*}
    \left( 1 - \frac{\mu \bar{\beta}_T}{L} \right) ^i \epsilon_0 \leq \epsilon.
\end{align*}
Taking the logarithm  on both sides of the inequality and using the fact that $-1/\log(1-x)\leq 1/x$ for $0<x\leq 1$, we have 
\begin{align*}
    i \geq  \frac{1}{\bar{\beta}_T}\kappa \log\left(  \frac{\epsilon_0}{\epsilon} \right).
\end{align*}

\section{Proof of Theorem \ref{thm:mutlinode_SS_SGD}} \label{app:thm:mutlinode_SS_SGD}

We prove the iteration complexities of multi-node stochastic gradient descent with
stochastic sparsification in Equation \eqref{eqn:CompressedSGD}.  We begin by introducing three useful lemmas for our analysis.

\begin{lemma}\label{lemma:trickBoundedVariance}
Let $\{x^i\}_{i\in\mathbb{N}}$ be the iterates  generated by Algorithm \eqref{eqn:CompressedSGD} and suppose that there exists  $\bar{\omega}_T$ such that  $\omega_j^i(T_j^i)\geq \bar{\omega}_T$ where $\omega_j^i$ is the sparsification parameter of node $j$ at iteration $i$. Then,
\begin{align*}
    \mathbf{E}\left\| \frac{1}{n}\sum_{j=1}^n Q_{T^i_j} \left( g_j(x^i;\xi_j^i) \right) \right\|^2 \leq (2/\bar\omega_T) \left( \mathbf{E}\| \nabla F(x^i)\|^2 + \sigma^2 \right).
\end{align*}
\end{lemma}
\begin{proof}
Since $\mathbf{E}\| Q_{T_j^i}(g_j(x^i;\xi_j^i)) \|^2 = \| g_j(x^i;\xi_j^i) \|^2/\omega_j^i(T_j^i)$, by using Cauchy-Scwartz's inequality and by the fact that $\omega_j^i(T_j^i)\geq \bar{\omega}_T$ where $\omega_j^i$ is the sparsification level of node $j$ at iteration $i$, we have
\begin{align*}
    \mathbf{E}\left\| \frac{1}{n}\sum_{j=1}^n Q_{T^i_j} \left( g_j(x^i;\xi_j^i) \right) \right\|^2 \leq \frac{1}{\bar\omega_T}\frac{1}{n} \sum_{j=1}^n \mathbf{E} \left\|   g_j(x^i;\xi_j^i)  \right\|^2. 
\end{align*}
After utilizing the inequality $\| x+y\|^2 \leq 2\|x\|^2 + 2\|y\|^2$ with $x =  g_j(x^i;\xi_j^i)  -\nabla F(x^i)$ and $y = \nabla F(x^i)$,
\begin{align*}
    \mathbf{E}\left\| \frac{1}{n}\sum_{j=1}^n Q_{T^i_j} \left( g_j(x^i;\xi_j^i) \right) \right\|^2  \leq \frac{2}{\bar\omega_T n} \sum_{j=1}^n  \left( T + \mathbf{E}\| \nabla F(x^i)\|^2 \right).
\end{align*}
where $T = \mathbf{E} \left\|   g_j(x^i;\xi_j^i)  -\nabla F(x^i) \right\|^2$.
By the bounded variance assumption  (i.e. $\mathbf{E}\| g_j(x;\xi_j) - \nabla F(x)\|^2\leq\sigma^2$), we complete the proof.

\end{proof}

\begin{lemma}\label{lemma:ProofForNC}
Suppose that each component function $f_j(\cdot)$ is $L$-smooth.
Let $\{x^i\}_{i\in\mathbb{N}}$ be the iterates  generated by Algorithm \eqref{eqn:CompressedSGD} and assume that there exists $\bar{\omega}_T$ such that  $\omega_j^i(T_j^i)\geq \bar{\omega}_T$ where $\omega_j^i$ is the sparsification parameter of node $j$ at iteration $i$. Then,
\begin{align*}
   \mathbf{E} F(x^{i+1}) 
    &\leq \mathbf{E} F(x^i) - \left(\gamma - (L/\bar\omega_T) \gamma^2 \right)\mathbf{E} \| \nabla F(x^i) \|^2   + (L/\bar\omega_T) \gamma^2 \sigma^2.
\end{align*}
\end{lemma}
\begin{proof}
By Cauchy-Scwartz's inequality, we can easily show that $F(x)=(1/n)\sum_{j=1}^n f_j(x)$ is also $L$-smooth. From the smoothness assumption (Lemma 1.2.3. of \cite{nesterov2018lectures}) and Equation \eqref{eqn:CompressedSGD}, 
\begin{align*}
    F(x^{i+1}) \leq F(x^i) - \gamma \left\langle \nabla F(x^i) , g^i  \right\rangle  + \frac{L\gamma^2}{2} \left\| g^i \right\|^2
\end{align*}
where $g^i = (1/n)\sum_{j=1}^n Q_{T^i_j} \left( g_j(x^i;\xi_j^i) \right)$.
Taking the expectation, and using Lemma \ref{lemma:trickBoundedVariance} and the unbiased properties of stochastic gradient $g_j(\cdot)$ and stochastic sparsification $Q_{T}(\cdot)$, we have
\begin{align*}
    \mathbf{E} F(x^{i+1}) 
    &\leq \mathbf{E} F(x^i) - \left(\gamma - (L/\bar\omega_T) \gamma^2 \right)\mathbf{E} \| \nabla F(x^i) \|^2   + (L/\bar\omega_T) \gamma^2 \sigma^2.
\end{align*}
\end{proof}

\begin{lemma}\label{lemma:MainCSC}
Suppose that each component function $f_j(\cdot)$ is $L$-smooth and $F$ is convex.
Let $\{x^i\}_{i\in\mathbb{N}}$ be the iterates  generated by Algorithm \eqref{eqn:CompressedSGD} and assume that there exists $\bar{\omega}_T$ such that  $\omega_j^i(T_j^i)\geq \bar{\omega}_T$ where $\omega_j^i$ is the sparsification parameter of node $j$ at iteration $i$. Then,
\begin{align*}
    \mathbf{E} \Vert x^{i+1} - x^\star \Vert^2 
    &\leq \mathbf{E} \Vert x^{i} - x^\star \Vert^2  + 2\gamma^2\sigma^2/\bar\omega_T - 2(\gamma - L\gamma^2/\bar\omega_T )\mathbf{E} \left\langle \nabla F(x^i), x^i - x^\star \right\rangle.
\end{align*}
\end{lemma}
\begin{proof}
From the definition of the Euclidean norm and Equation \eqref{eqn:CompressedSGD}, we have 
\begin{align*}
    \Vert x^{i+1} - x^\star \Vert^2 = \Vert x^{i} - x^\star \Vert^2 - 2\gamma \left\langle g^i, x^i - x^\star \right\rangle + \gamma^2 \| g^i \|^2,
\end{align*}
where $g^i = (1/n)\sum_{j=1}^n Q_{T^i_j} \left( g_j(x^i;\xi_j^i) \right)$. Taking the expectation,  and using Lemma \ref{lemma:trickBoundedVariance} and the unbiased properties of stochastic gradient $g_j(\cdot)$ and stochastic sparsification $Q_{T}(\cdot)$, we have
\begin{align*}
    \mathbf{E} \Vert x^{i+1} - x^\star \Vert^2 
    &\leq \mathbf{E} \Vert x^{i} - x^\star \Vert^2+ (2/\bar\omega_T)\gamma^2\sigma^2 \\ 
    &\hspace{0.4cm} - 2\gamma \mathbf{E} \left\langle \nabla F(x^i), x^i - x^\star \right\rangle  \\
    &\hspace{0.4cm}+ (2/\bar\omega_T)\gamma^2 \mathbf{E}\| \nabla F(x^i) \|^2 
\end{align*}
Since $F(\cdot)$ is $L$-smooth, i.e. for $x\in\mathbb{R}^d$
$$
\| \nabla F(x) - \nabla F(y) \|^2 \leq L \left\langle \nabla F(x) -  \nabla F(y) , x -y \right\rangle,
$$
applying this inequality with $x=x^i$ and $y=x^\star$ into the main result and 
recalling that $\nabla F(x^\star) =0$ we complete the proof. 
\end{proof}

Now, we prove the main results for Algorithm \eqref{eqn:CompressedSGD}.

\subsection*{Proof of Theorem \ref{thm:mutlinode_SS_SGD}-1.}
If $\gamma < \bar\omega_T/L$, rearranging the terms from  Lemma \ref{lemma:trickBoundedVariance},  we get
\begin{align*}
   \mathbf{E} \| \nabla F(x^i) \|^2  & \leq  \frac{1}{\gamma - (L/\bar\omega_T) \gamma^2}\left( \mathbf{E} F(x^i) -  \mathbf{E} F(x^{i+1}) \right) + \frac{(L/\bar\omega_T) \gamma }{1 - (L/\bar\omega_T) \gamma}\sigma^2.
\end{align*}
Since $\min_{l\in[0,i-1]} \mathbf{E}\| \nabla F(x^l) \|^2 \leq \sum_{l=0}^{i-1} \mathbf{E}\| \nabla F(x^l) \|^2/i$, we obtain
\begin{align*}
   &\min_{l\in[0,i-1]} \mathbf{E}\| \nabla F(x^l) \|^2  \leq  \frac{1}{ i[ \gamma - (L/\bar\omega_T) \gamma^2]}\left( \mathbf{E} F(x^0) -  \mathbf{E} F(x^{i}) \right)  +T,
\end{align*}
where $T=\sigma^2\cdot{(L/\bar\omega_T) \gamma }/{[1 - (L/\bar\omega_T) \gamma]}$.
By the fact that $F(x)\geq F(x^\star)$ for $x\in\mathbb{R}^d$, we have 
\begin{align*}
   &\min_{l\in[0,i-1]} \mathbf{E}\| \nabla F(x^l) \|^2 \leq  \frac{1}{ i[ \gamma - (L/\bar\omega_T) \gamma^2]}\epsilon_0  +T,
\end{align*}
where $\epsilon_0=F(x^0)-F(x^\star)$. 
If the step-size is 
$$\gamma = \frac{\bar\omega_T}{2L}\frac{1}{2\sigma^2/\epsilon+1},$$
then clearly $\gamma < \bar\omega/L$ and 
$$
 \frac{(L/\bar\omega_T) \gamma }{1 - (L/\bar\omega_T) \gamma}\sigma^2 \leq \frac{\epsilon}{2}.
$$
From  Lemma \ref{lemma:ProofForNC}, Algorithm \eqref{eqn:CompressedSGD} reaches $\min_{l\in[0,i-1]} \mathbf{E}\| \nabla F(x^l) \|^2 \leq\epsilon$ for the number of iterations $i$ which fulfills
$$
\frac{1}{i} \frac{2L\epsilon_0}{\bar\omega_T}(2\sigma^2/\epsilon+1) \cdot \frac{4\sigma^2/\epsilon+2}{4\sigma^2/\epsilon+1} \leq \frac{\epsilon}{2}. 
$$
Since $\left( 4\sigma^2/\epsilon+ 2 \right)/\left(4\sigma^2/\epsilon+1\right) \leq 2$, 
the main condition can be rewritten  equivalently as
\begin{align*}
    i \geq \frac{4}{\bar\omega_T} \left( 1+ \frac{2\sigma^2}{\epsilon} \right) \cdot \frac{2L\epsilon_0}{\epsilon}.
\end{align*}

\subsection*{Proof of Theorem \ref{thm:mutlinode_SS_SGD}-2.}
If $\gamma < \bar\omega_T/L$, by Lemma \ref{lemma:MainCSC} and by using the convexity of $F$, i.e. $ \left\langle \nabla F(x), x - x^\star \right\rangle \geq F(x)-F(x^\star)$ for $x\in\mathbb{R}^d$ we get 
\begin{align*}
    \mathbf{E} \Vert x^{i+1} - x^\star \Vert^2 
    &\leq \mathbf{E} \Vert x^{i} - x^\star \Vert^2   + 2\gamma^2\sigma^2/\bar\omega_T - 2(\gamma - L\gamma^2/\bar\omega_T )\mathbf{E} [F(x^i) - F(x^\star)]
\end{align*}
By rearranging the terms and using the fact that $F$ is convex, i.e. $F( \sum_{l=0}^{i-1} x^l) \leq \sum_{l=0}^{i-1} F(x^l)$, we then have 
\begin{align*}
     \mathbf{E} \left[ F\left(\frac{1}{i}\sum_{l=0}^{i-1} x^l  \right)- F(x^\star) \right] 
    & \leq \frac{1}{i} \sum_{l=0}^{i-1} \mathbf{E}[ F(x^l) - F(x^\star)] \\
    & \leq \frac{1}{i}\frac{1}{2\gamma}\frac{\epsilon_0}{1-(L/\bar\omega_T)\gamma}  + \frac{\gamma\sigma^2/\bar\omega_T}{1-(L/\bar\omega_T)\gamma},
\end{align*}
where $\epsilon_0 = \| x^0 - x^\star \|^2$. The last inequality follows from the cancellations of the telescopic series the fact that $\| x\|^2\geq 0$ for $x\in\mathbb{R}^d$. If the step-size is
$$
\gamma = \frac{\bar\omega_T}{2}\frac{1}{2\sigma^2/\epsilon +L}, 
$$
then clearly $\gamma  < \bar\omega_T/L$ and 
$$
\frac{\gamma\sigma^2/\bar\omega_T}{1-(L/\bar\omega_T)\gamma} \leq \frac{\epsilon}{2}.
$$
To reach $ \mathbf{E} \left[ F\left(\sum_{l=0}^{i-1} x^l /i \right)- F(x^\star) \right] \leq\epsilon$, Algorithm \eqref{eqn:CompressedSGD} needs the number of iterations $i$
satisfying
%
%
%
%
\begin{align*}
    \frac{1}{i} \frac{1}{\bar\omega_T} \frac{2( 2\sigma^2/\epsilon+L)^2}{4\sigma^2/\epsilon+ L} \epsilon_0 \leq \frac{\epsilon}{2}.
\end{align*}
Since $\left( 4\sigma^2/\epsilon+ 2L \right)/\left(4\sigma^2/\epsilon+L\right) \leq 2$, the main condition can be rewritten equivalently as 
\begin{align*}
    i \geq \frac{2}{\bar\omega_T}\left( 1+ \frac{2\sigma^2}{\epsilon L}\right)\frac{2L\epsilon_0}{\epsilon}
\end{align*}
\subsection*{Proof of Theorem \ref{thm:mutlinode_SS_SGD}-3. }
If $\gamma< \bar \omega_T/L$, by Lemma \ref{lemma:ProofForNC} and the  strong convexity of $F(\cdot)$, i.e. $\| \nabla F(x) \|^2 \geq 2\mu [F(x) - F(x^\star)]$ for $x\in\mathbb{R}^d$ we have 
\begin{align*}
    \mathbf{E} [F(x^{i+1}) - F(x^\star)] 
    &\leq \rho \mathbf{E}[F(x^i) - F(x^\star)] + \frac{L}{\bar\omega_T} \gamma^2\sigma^2, \intertext{where} 
    \rho & = 1 - 2\mu\left(\gamma - \frac{L}{\bar\omega_T}\gamma^2\right).
\end{align*}
By applying the inequality recursively, we get 
\begin{align*}
    \mathbf{E} [F(x^{i+1}) - F(x^\star)] \leq \rho^i \epsilon_0 + \frac{L}{2\mu\bar\omega_T}\frac{\gamma}{1-L\gamma/\bar\omega_T}\sigma^2,
\end{align*}
where $\epsilon_0=F(x^0)-F(x^\star)$. 
If the step-size is 
\begin{align*}
    \gamma = \frac{\bar\omega_T}{2}\frac{1}{2\sigma^2/(\mu\epsilon) +L},
\end{align*}
then clearly $\gamma<\bar\omega_T/L$ and 
\begin{align*}
 \frac{L}{2\mu\bar\omega_T}\frac{\gamma}{1-L\gamma/\bar\omega_T}\sigma^2 \leq \frac{\epsilon}{2}.
\end{align*}

To reach $ \mathbf{E} \| x^i - x^\star \|^2 \leq\epsilon$, Algorithm \eqref{eqn:CompressedSGD} needs the number of iterations $i$ which satisfies
%
%
\begin{align*}
  \left( 1- \frac{\mu}{2}\cdot\frac{\bar\omega_T({4\sigma^2}/{(\mu\epsilon)} +L)}{\left( {2\sigma^2}/{(\mu\epsilon)} +L\right)^2} \right)^i \epsilon_0 \leq \frac{\epsilon}{2}    
\end{align*}
Taking the logarithm on both sides, and utilizing the fact that $-1/\log(1-x)\leq 1/x$ for $0<x\leq 1$, we have 
$$
i \geq \frac{2}{\bar\omega_T \mu} \frac{\left( {2\sigma^2}/{(\mu\epsilon)} +L\right)^2}{{4\sigma^2}/{(\mu\epsilon)} +L} \log\left( \frac{2\epsilon_0}{\epsilon} \right).
$$
Since ${ \left( {4\sigma^2}/{(\mu\epsilon)} +2L\right)}/\left({{4\sigma^2}/{(\mu\epsilon)} +L}\right) \leq 2$, the main condition can be rewritten equivalently as 
$$
i \geq \frac{2}{\bar\omega_T} \kappa \left( 1+ \frac{2\sigma^2}{\mu\epsilon L} \right) \log\left( \frac{2\epsilon_0}{\epsilon} \right).
$$

\section{Discussions on Optimizing Parameters of Stochastic Sparsification} \label{app:DiscussionSM}

 In this section, we show how to tune the parameters of stochastic sparsification $p$ to maximize the descent direction. 
 In optimization formulation, for a fixed sparsity budget $T$ we obtain the optimal probabilities $p^{\star}$ by solving the following problem~\cite{wang2018atomo}
 %
\begin{equation}\label{eqn:minVariance}
    \begin{aligned}
    \mathop{\text{maximize}}\limits_{p\in [0,1]^d} & \quad \omega_p(T) 
    \\ \text{subject to} & \quad  \sum_{j=1}^d p_j = T.
    \end{aligned}
\end{equation}
Here, $\omega_p(T)$ can be rewritten as 
\begin{align*}
    \omega_p(T) = \frac{\|g\|^2}{\mathbf{E}\| Q_{T,p}(g) \|^2} = \frac{\|g\|^2}{\sum_{i=1}^d \frac{1}{p_j} g_j^2}.
\end{align*}

This minimization problem \eqref{eqn:minVariance} has the optimal solution $p^\star = (p_1^\star,p_2^\star,\ldots,p_d^\star)$ on the form \cite{wang2018atomo} 
\begin{align*}
    p_i^\star = \begin{cases} 1 & \text{if } i=1,\ldots,n_s \\ {\vert \lambda_i \vert (s - n_s)}/{\sum_{j=n_s+1}^n \vert \lambda_j \vert } & \text{if } i =n_s+1,\ldots,n  \end{cases}
\end{align*}
where $n_s$ is selected such that $p_i^\star$ is bounded above by $1$.
This observation results in many algorithms to compute the optimal probabilities (see e.g., Algorithm 1 in \cite{wang2018atomo}).

\section{Descent Lemma for Multi-node Gradient Methods with Stochastic Sparsification}\label{app:lemma:MultinodeSS_GD}

We include a descent lemma for distributed compressed gradient methods with stochastic sparsification in Equation \eqref{eqn:CompressedSGD}, which are analogous to single-node gradient methods using stochastic sparsification.  

\begin{lemma}
Consider the minimization problem over $F(x)=(1/n)\sum_{j=1}^n f_j(x)$ where each $f_j(\cdot)$ is  (possibly non-convex) $L$-smooth. Let $\omega(T)=(\sum_{j=1}^n 1/[n\omega_j(T_j)] )^{-1}$, where $\omega_j(T_j)$ is the sparsification parameter of node $j$. Suppose that  
\begin{align}\label{eqn:Lemma:MultinodeSS_GD}
x^{+} = x- \gamma \frac{1}{n}\sum_{j=1}^n Q_{T_j}( g_j(x;\xi_j) ),
\end{align}
where each $g_j(x;\xi_j)$  is unbiased and has variance with respect to $\nabla F(x)$ bounded by $\sigma^2$. If $\gamma = \omega(T)/(2L)$, then 
    \begin{align*}
    \mathbf{E} F(x^{+}) \leq  \mathbf{E} F(x) - \frac{\omega(T)}{4L}\mathbf{E}\| \nabla F(x)\|^2 + \frac{\omega(T)}{4L}\sigma^2.
    \end{align*}

\end{lemma}

\begin{proof}
By Cauchy-Schwartz's inequality, we can show that $F(x)=(1/n)\sum_{j=1}^n f_j(x)$ is also $L$-smooth. Using the smoothness assumption (Lemma 1.2.3. of \cite{nesterov2018lectures}) and Equation \eqref{eqn:CompressedSGD}, we have 
\begin{align*}
    F(x^+) \leq F(x) - \gamma \frac{1}{n}\sum_{j=1}^n \langle \nabla F(x), Q_{T_j} (g_j(x;\xi_j)) \rangle + \frac{L\gamma^2}{2} \frac{1}{n}\sum_{j=1}^n \|  Q_{T_j} (g_j(x;\xi_j))\|^2.
\end{align*}

Since $\omega_j(T_j)$ is defined by 
\begin{align*}
    \omega_j(T_j)=    \frac{|| g_j(x;\xi_j)||^2}{ \mathbf{E} ||{Q}_{T_j}(g(x;\xi_j)) ||^2},
\end{align*}
we can easily show that 
\begin{align}\label{eqn:DescentLemmaSS_GD_main}
    \frac{1}{n}\sum_{j=1}^n \|  Q_{T_j} (g_j(x;\xi_j))\|^2 \leq \frac{2}{\omega(T)}\| \nabla F(x) \|^2 + \frac{2}{\omega(T)}\sigma^2,
\end{align}
where $\omega(T)=(\sum_{j=1}^n 1/[n\omega_j(T_j)])^{-1}$ and $\omega_j(T_j)$ is the sparsification parameter of node $j$. This result follows from using the inequality $\| x+y\|^2 \leq 2\|x\|^2+2\|y\|^2$ with $x= g_j(x;\xi_j) - \nabla F(x)$ and $y=\nabla F(x)$, and from the fact that $\mathbf{E}\|g_j(x;\xi_j)-\nabla F(x)\|^2\leq \sigma^2$.

Next, by taking the expectation and then using unbiased property of stochastic gradients, the fact that $F(x)=(1/n)\sum_{j=1}^n f_j(x)$, and Inequality \eqref{eqn:DescentLemmaSS_GD_main},  we have 
\begin{align*}
    \mathbf{E} F(x^+) \leq \mathbf{E} F(x) - \left(\gamma - L\gamma^2/\omega(T) \right) \mathbf{E} \| \nabla F(x) \|^2 + L\gamma^2 \sigma^2/\omega(T).
\end{align*}
Choosing $\gamma = \omega(T)/(2L)$, we complete the proof.

\end{proof}

\section{Additional Experiments on Logistic Regression over \texttt{URL}
}

\begin{figure}
    \centering
    \includegraphics[width=\textwidth]{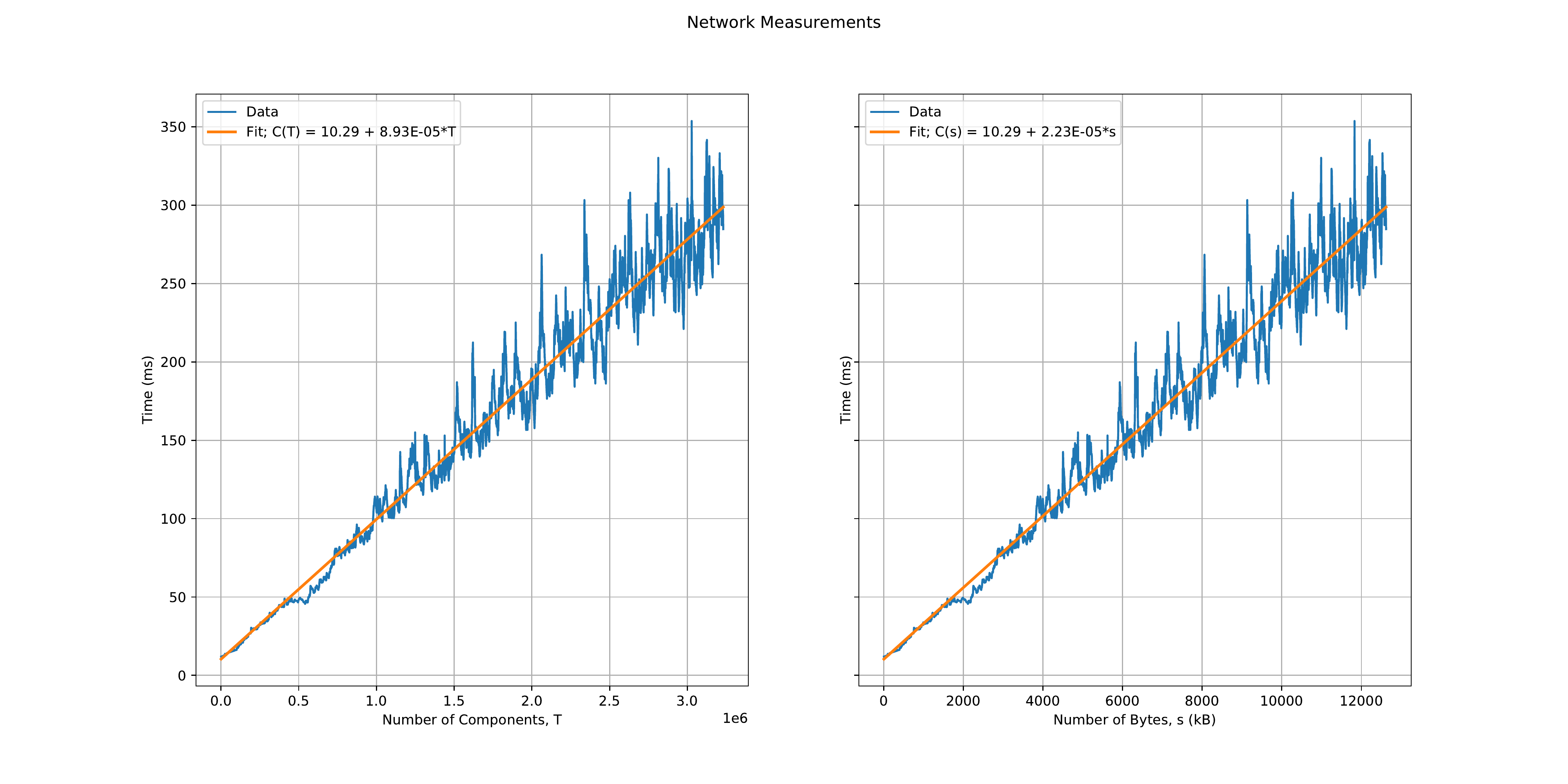}
    \caption{Time for communicating the vector with sparsity budget $T$ (left) and its associated bytes (right).}
    \label{fig:network-model}
\end{figure}

\begin{figure}
    \centering
    \includegraphics[width=\textwidth]{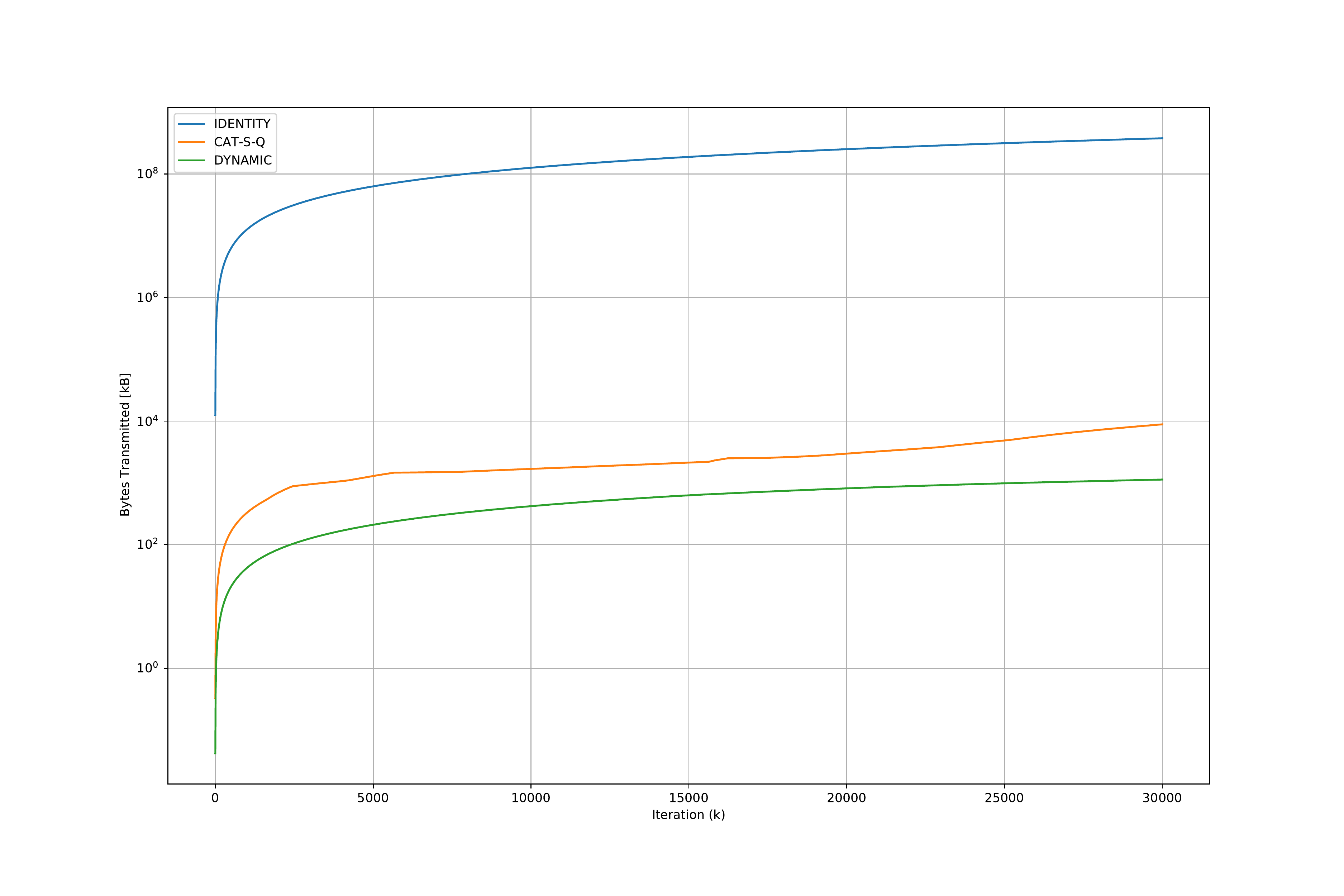}
    \caption{Gradient transmission (bytes) in each iteration $k$ for full gradient descent (identity), CAT S+Q, and Alistarh's S+Q (dynamic). }
    \label{fig:comm-efficiency}
\end{figure}

\begin{figure}
    \centering
    \includegraphics[width=\textwidth]{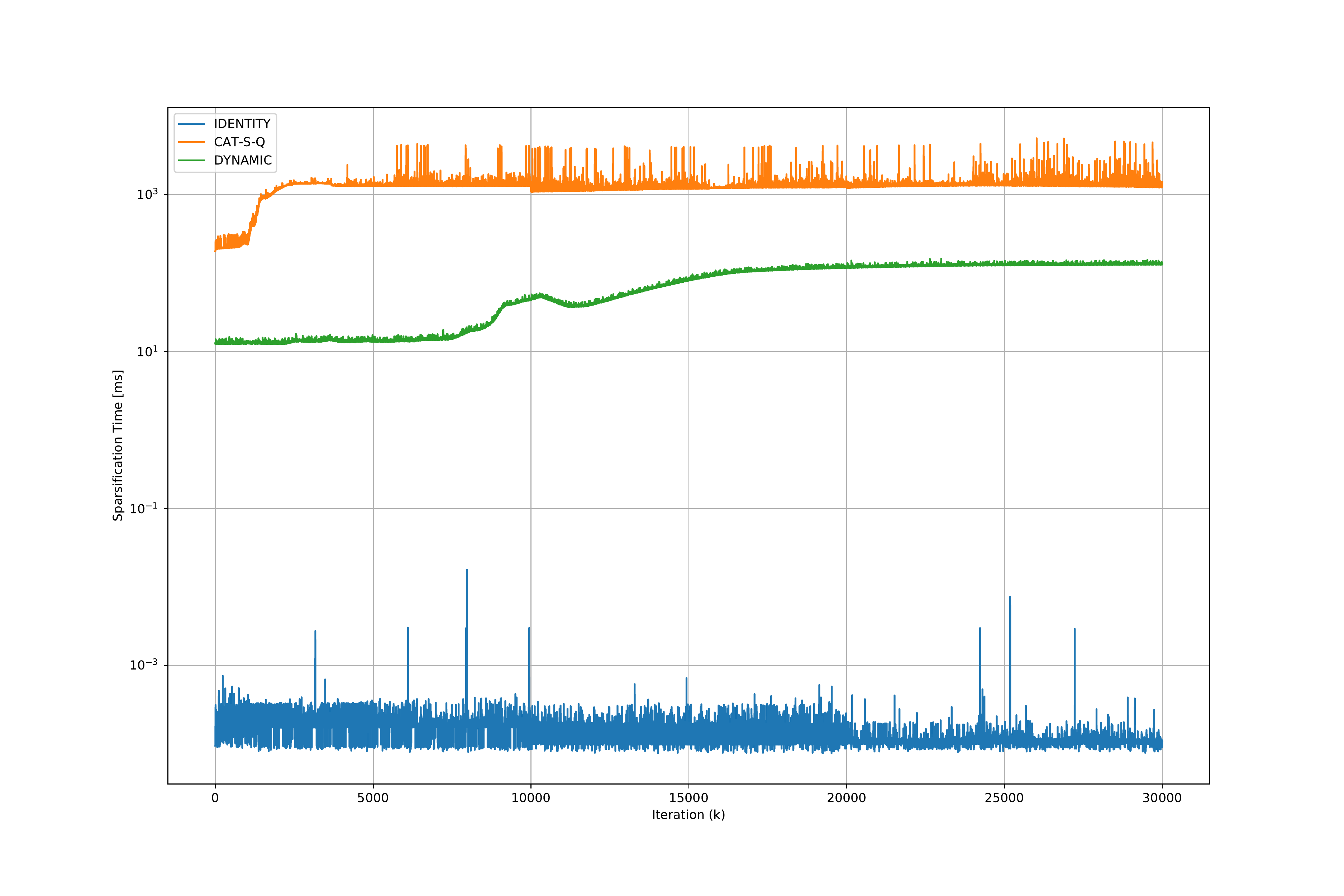}
    \caption{Sparsification time in each iteration $k$ for full gradient descent (identity), CAT S+Q, and Alistarh's S+Q (dynamic).}
    \label{fig:spar-time}
\end{figure}

\begin{figure}
    \centering
    \includegraphics[width=\textwidth]{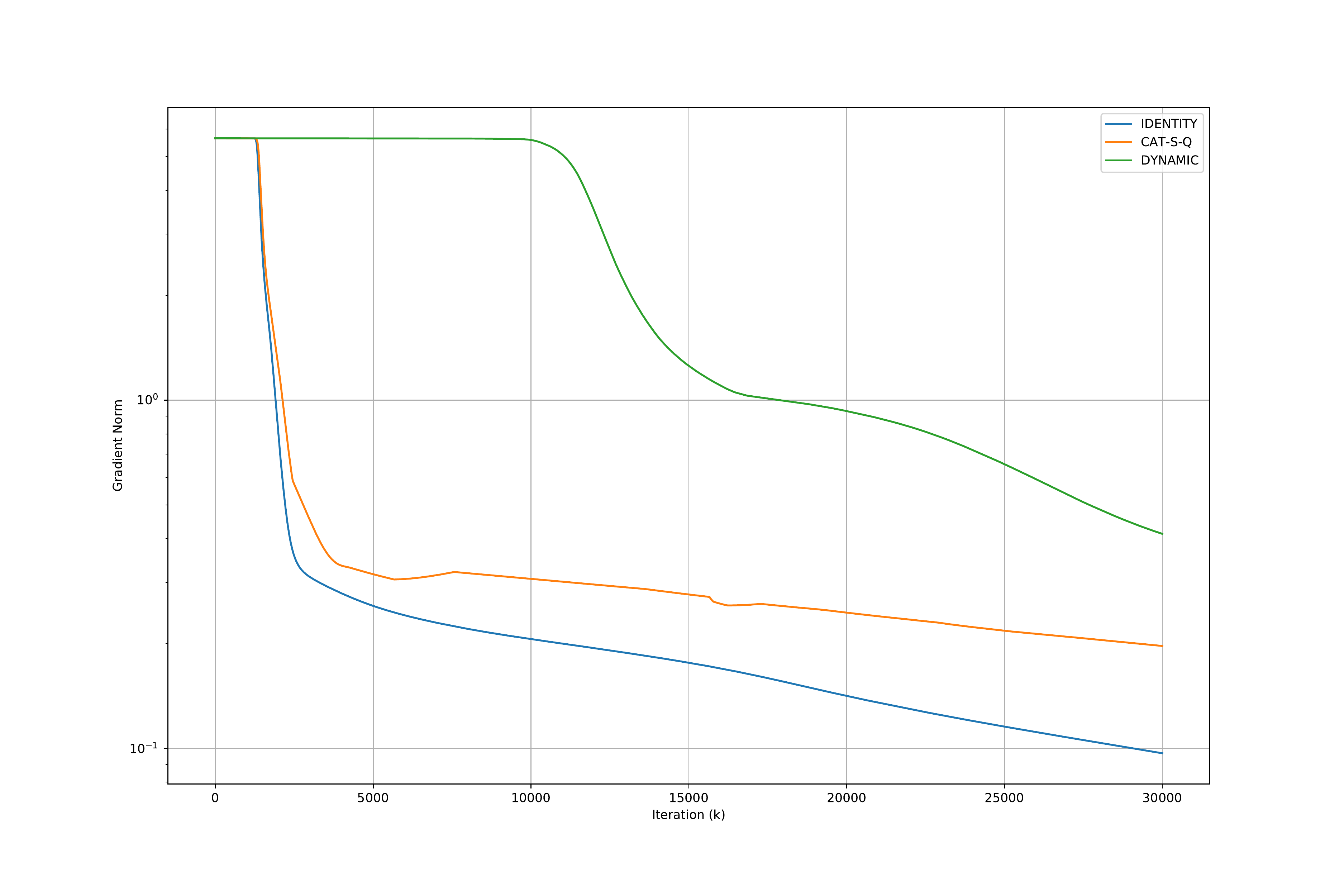}
    \caption{Convergence performance in the gradient norm $\| \nabla F(x^k) \|$ for full gradient descent (identity), CAT S+Q, and Alistarh's S+Q (dynamic).}
    \label{fig:grad-norm}
\end{figure}

\begin{figure}
    \centering
    \includegraphics[width=\textwidth]{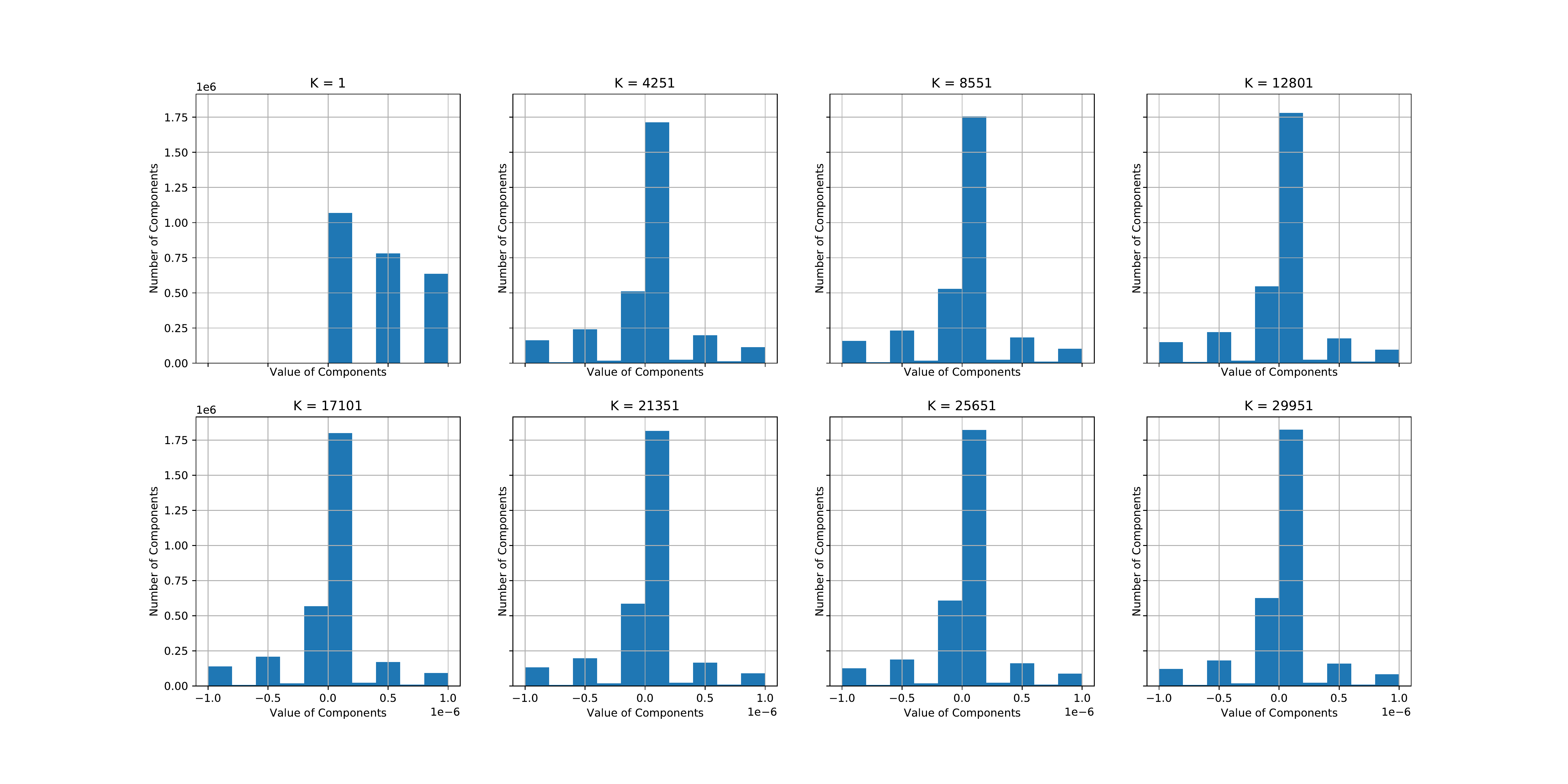}
    \caption{Histogram of the gradient elements when CAT S+Q is run for  $30,000$ iterations. }
    \label{fig:grad-hist}
\end{figure}

\begin{figure}
    \centering
    \includegraphics[width=0.5\textwidth]{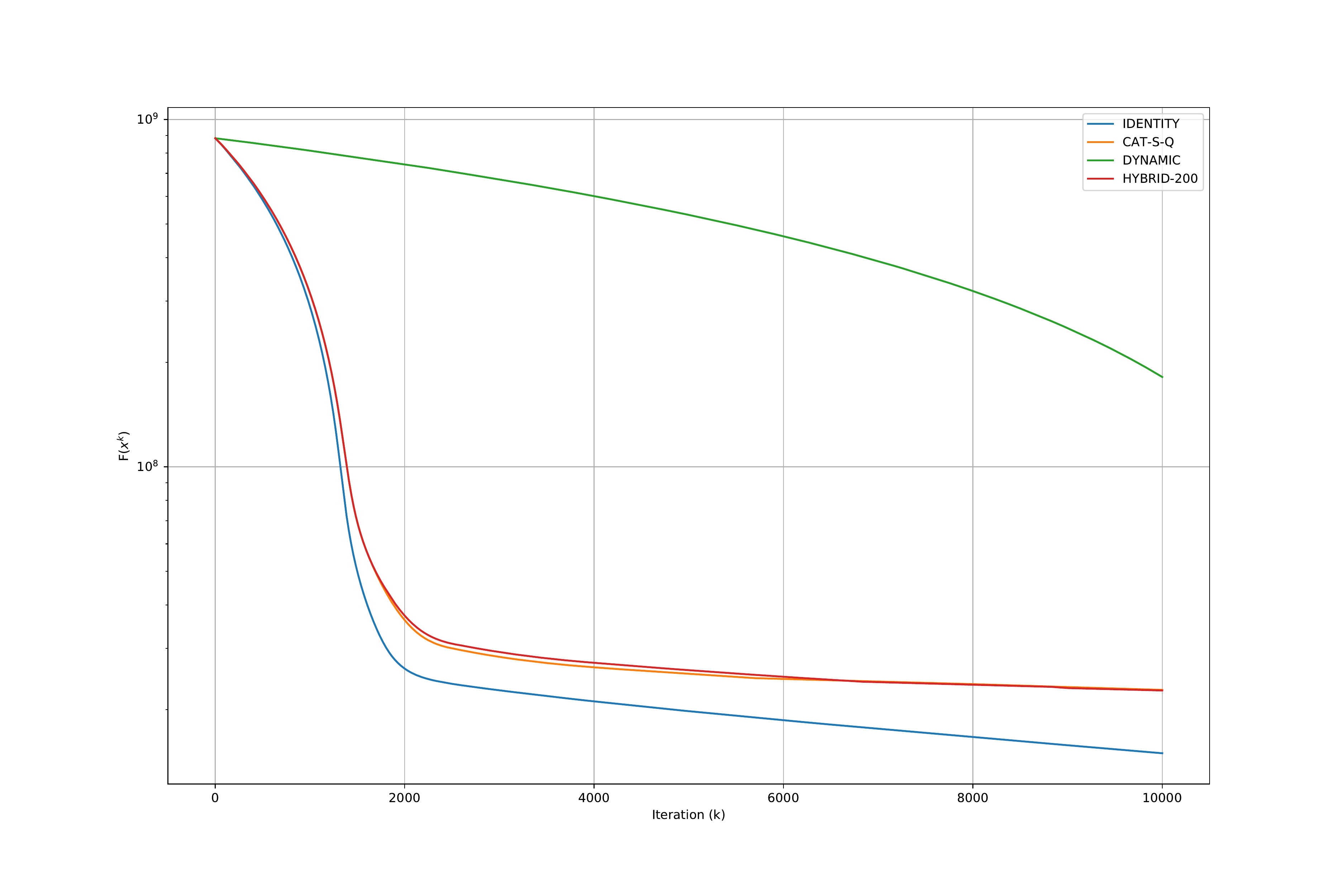}%
    \includegraphics[width=0.5\textwidth]{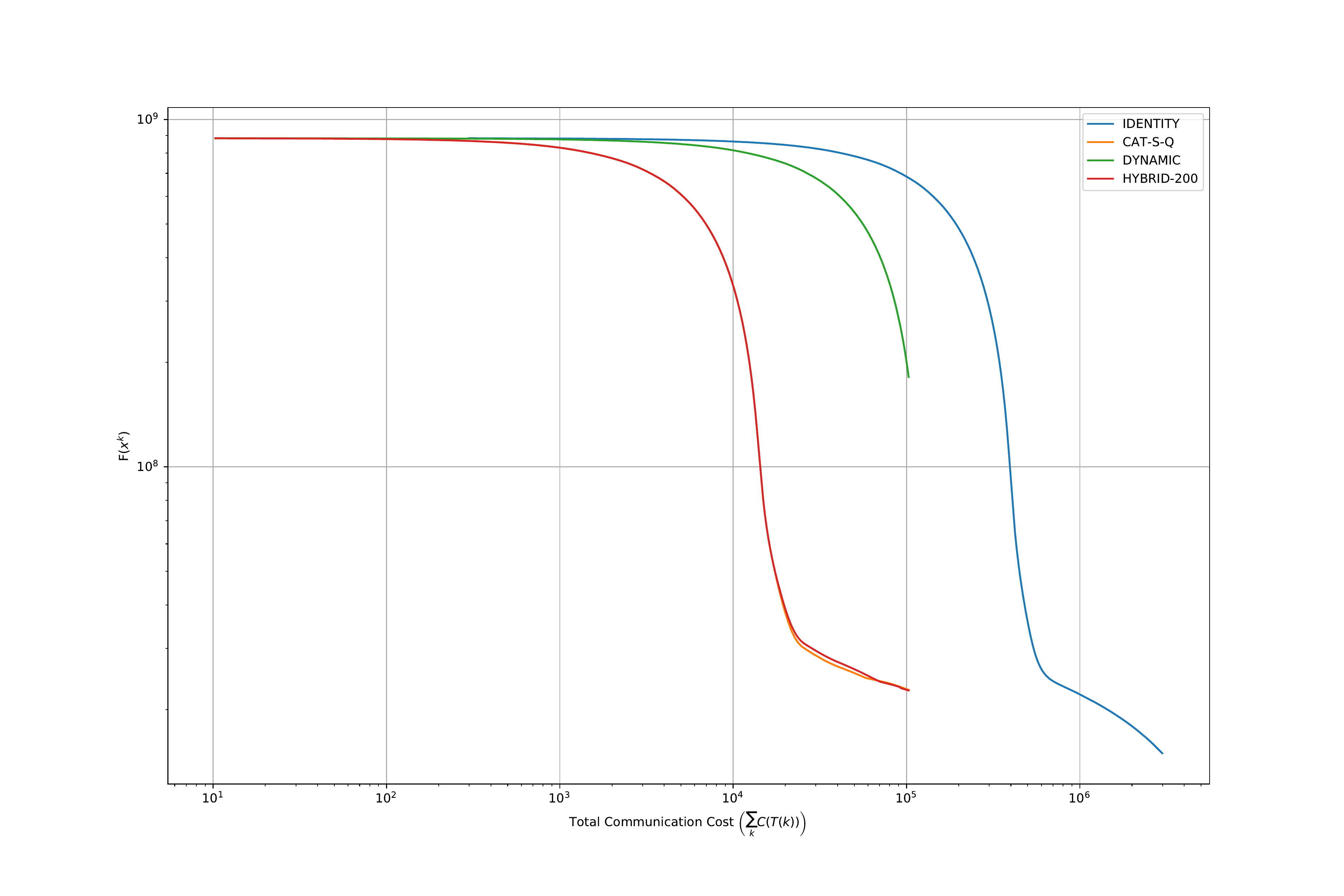}
    \caption{
    Performance with respect to iteration counts (left) and to communication costs (right) for different algorithms. We evaluated full gradient descent (identity), CAT S+Q, Alistarh's S+Q (dynamic) and the hybrid algorithm that uses CAT framework for every $200$ iterations (hybrid-200).}
    \label{fig:comp-hybrid-1}
\end{figure}

\begin{figure}
    \centering
    \includegraphics[width=0.5\textwidth]{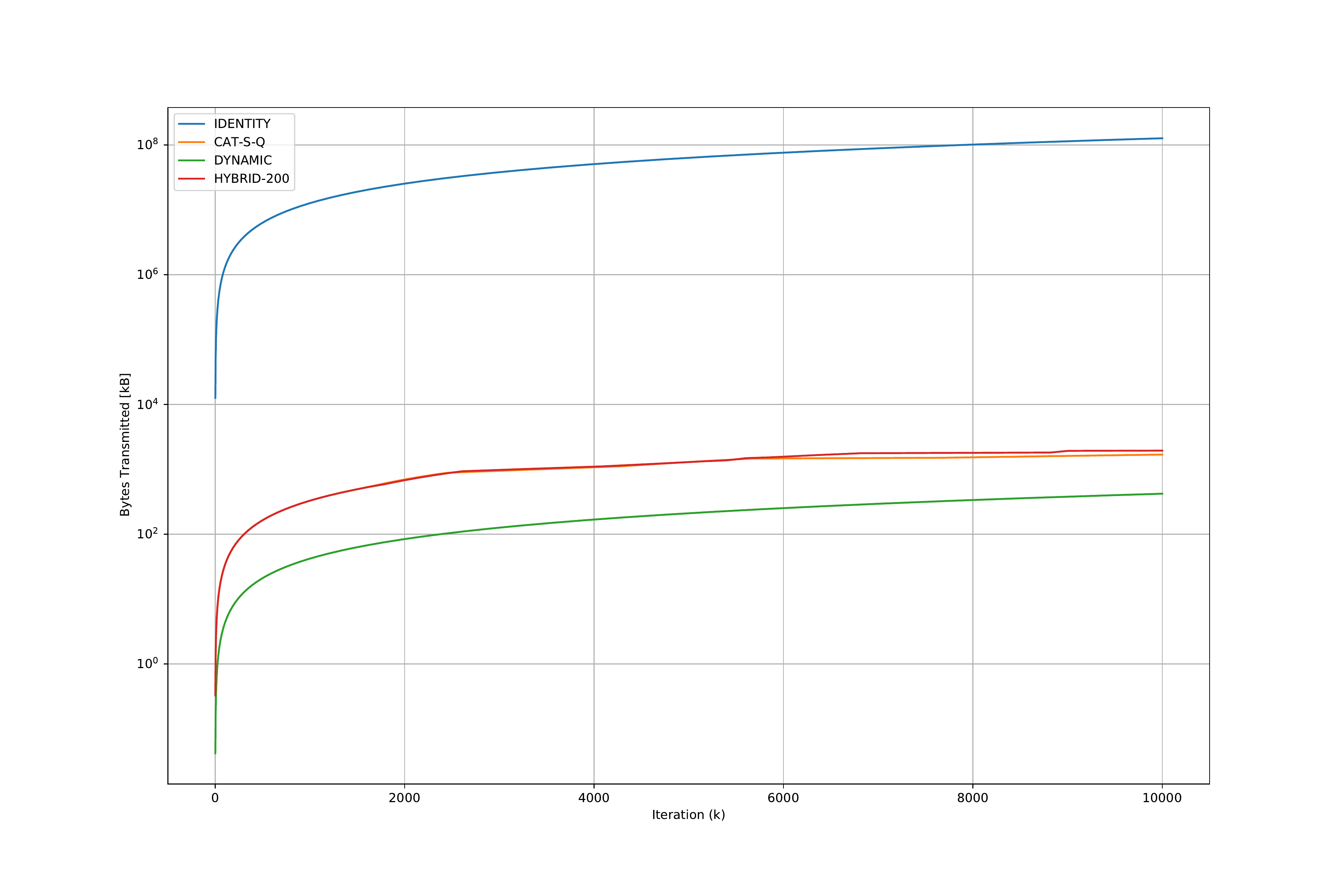}%
    \includegraphics[width=0.5\textwidth]{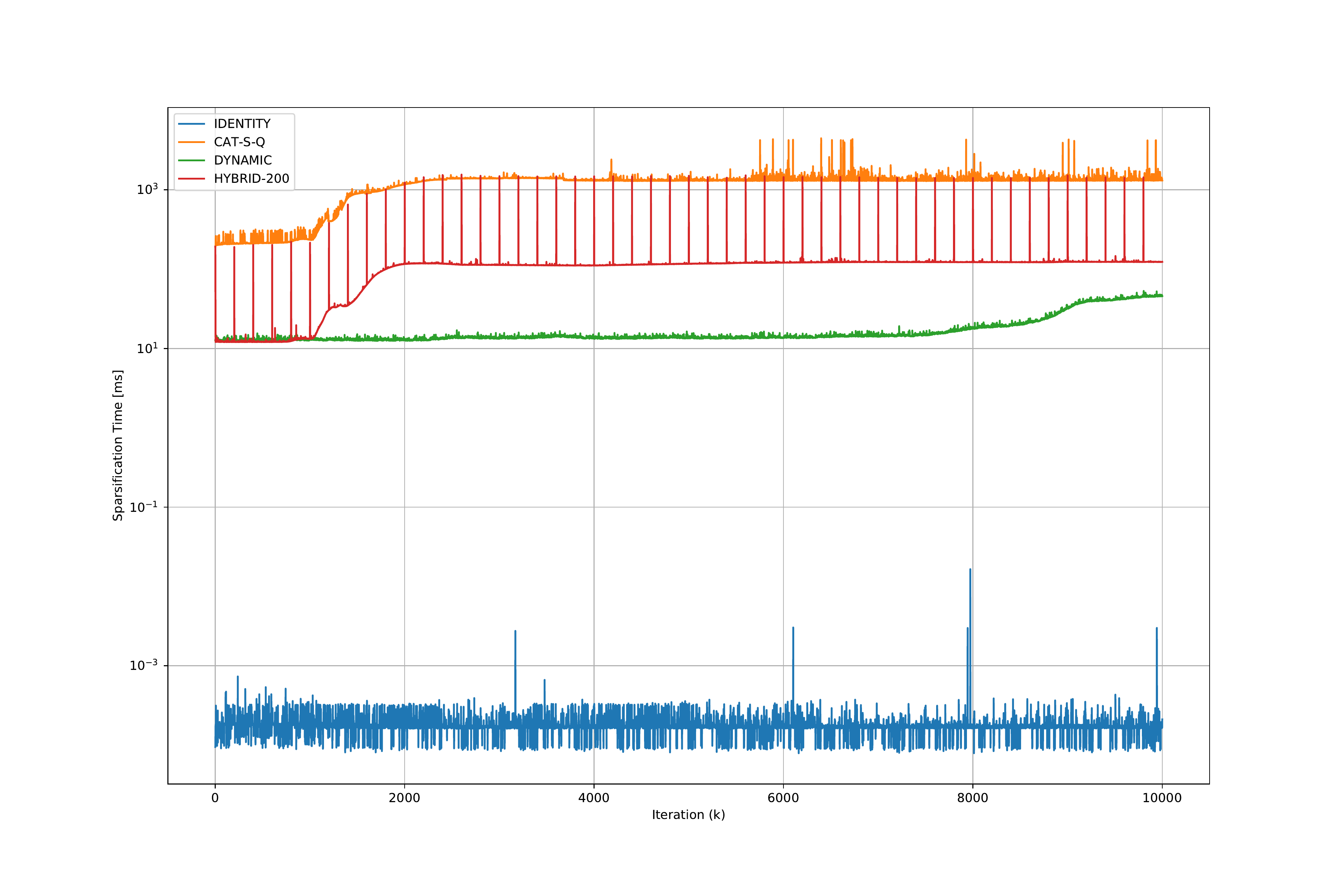}
    \caption{
    Gradient transmission in bytes (left)
    and  sparsification time in each iteration (right) for different algorithms. We evaluated full gradient descent (identity), CAT S+Q, Alistarh's S+Q (dynamic) and the hybrid algorithm that uses CAT framework for every $200$ iterations (hybrid-200).}
    \label{fig:comp-hybrid-2}
\end{figure}

In this section, we provide additional simulations of logistic regression problems over the \texttt{URL} data set with dimension $d=3.2\cdot 10^6$. We fit a linear communication model based on measurements, and benchmark our CAT framework and the heuristic from~\cite{alistarh2017qsgd} on gradient descent with dynamic sparsification together with quantization in the single-master, single-worker architecture.

In Figure \ref{fig:network-model} each blue data point represents the average of ten measurements of the end-to-end transmission time of a sparsified gradient with sparsity budget $T=\{1000, 2000, 3000, \ldots, d\}$. The orange lines demonstrate that an affine communication model is able to capture the communication cost. In retrospect, the affine behaviour should be expected, since we use the \texttt{ZMQ} library which initiates TCP communication once, and then reuses the communication together with buffers to optimize message transmission. Our ability to capture the communication cost (time) with an affine model indicates CAT that the framework could provide near-optimal performance in terms of communication time.
Similarly for energy-constrained applications in IoT devices we can indeed investigate how the energy spent is related to the information transmitted. Utilizing this characteristics, our CAT framework can communicate information efficiently with low energy costs.

To illustrate how the CAT framework reduces communication cost and wall-clock time to reach target solution accuracy, we compared CAT S+Q  and Alistarh's S+Q (dynamic), against full gradient descent. 
From Figure \ref{fig:comm-efficiency}, CAT S+Q  and Alistarh's S+Q reduce communication costs by 4 and 5 orders of magnitudes, respectively, compared to full gradient descent. %
However, we also observe that sparsification time of CAT S+Q is higher than Alistarh's S+Q by roughly an order of magnitude, as shown in Figure \ref{fig:spar-time}. Interestingly, the sparsification times for both algorithms increase as iteration counts grow. 
This happens because the sorting strategy in C++ leads to a worse performance especially when the gradient elements become more homogeneous. 
After running CAT S+Q for $30,000$ iterations, we observe in Figure \ref{fig:grad-hist} that 
gradient information after iteration $K=4251$ has very similar histograms. In this case, sorting from scratch at each iteration is inefficient.

Based on these observations,  we propose to let the CAT framework update the sparsity budget every $S$ iterations (rather than every iteration) and then keep the sparsity budget fixed in the iterates between CAT recalculations. 
%
%
Figure \ref{fig:comp-hybrid-1} shows that this hybrid heuristic with $S=200$ (i.e. rather infrequent updates) achieves only slightly worse loss function convergence with respect to iteration count and communication cost.
From Figure \ref{fig:comp-hybrid-2},
despite almost the same data transmission size in each iteration, CAT S+Q reduces the time to sparsify the gradients by roughly an order of magnitude.

\section{Additional Experiments for Sparsification+Quantization, and Stochastic Sparsification} \label{App:AE}

\begin{figure}[t]
\centering
   \includegraphics[width=0.47\textwidth]{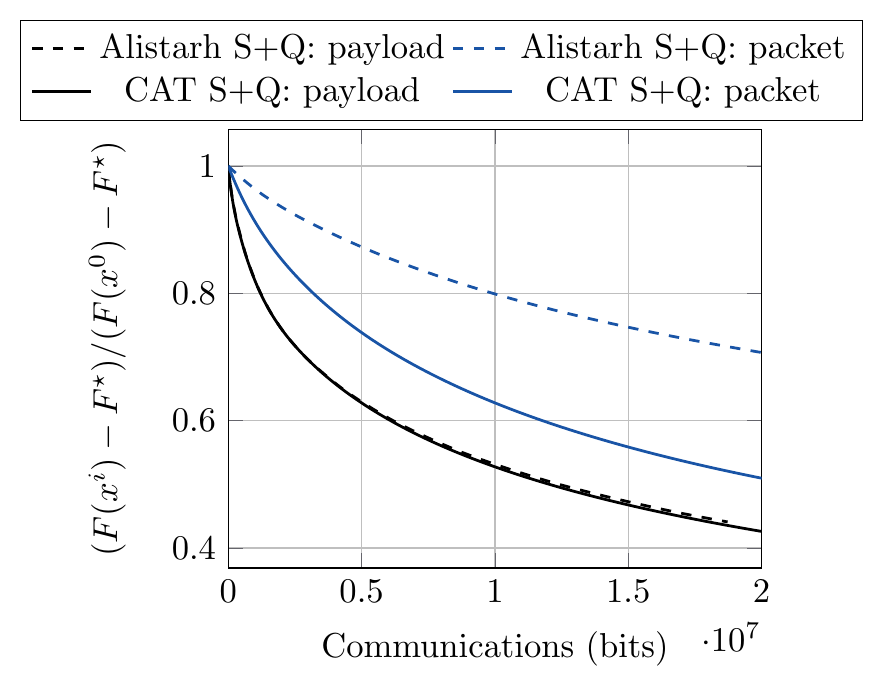}
    ~ 
    \caption{CAT sparsification + quantization on the \texttt{RCV1} data set.} 
    \label{fig:SQ_eternet}
\end{figure}

In this section, we include additional simulations that illustrate communication efficiency of gradient descent with our CAT frameworks for logistic regression problems on the \texttt{RCV1} data set. We implemented
communication aware algorithms using sparsification together with quantization, and using stochastic sparsification. 

For sparsification together with quantization, we benchmarked CAT against the dynamic tuning introduced in~\cite{alistarh2017qsgd}. 
%
%
The black lines of Figure ~\ref{fig:SQ_eternet} illustrate 
convergence results for the payload communication model, i.e., $C(T)=P^{\texttt{SQ}}(T)$ defined in Equation~\eqref{eq:Payload_SQ}. 
 The blue lines are the results for when $C(T)$ follows the packet model in Equation~\eqref{eq:COMM2} with $c_1=128$ bytes, $c_0=64$ bytes and $P_{\max}=128$ bytes. These lines indicate that if we only count the payload, then the two methods are comparable.  
 Our CAT tuning rule outperforms~\cite{alistarh2017qsgd} by only a small margin. This suggests that the heuristic in~\cite{alistarh2017qsgd} is quite communication efficient in the simplest communication model using only payload.  
 However, the heuristic rule is agnostic to the actual communication model. 
 Therefore, we should not expect it to perform well for general $C(T)$. 
 The blue lines show that the CAT is roughly two times more communication efficient than the the dynamic tuning rule in~\cite{alistarh2017qsgd} for the packet communication model.

\begin{figure}
    \centering
    \includegraphics[width=0.45\textwidth]{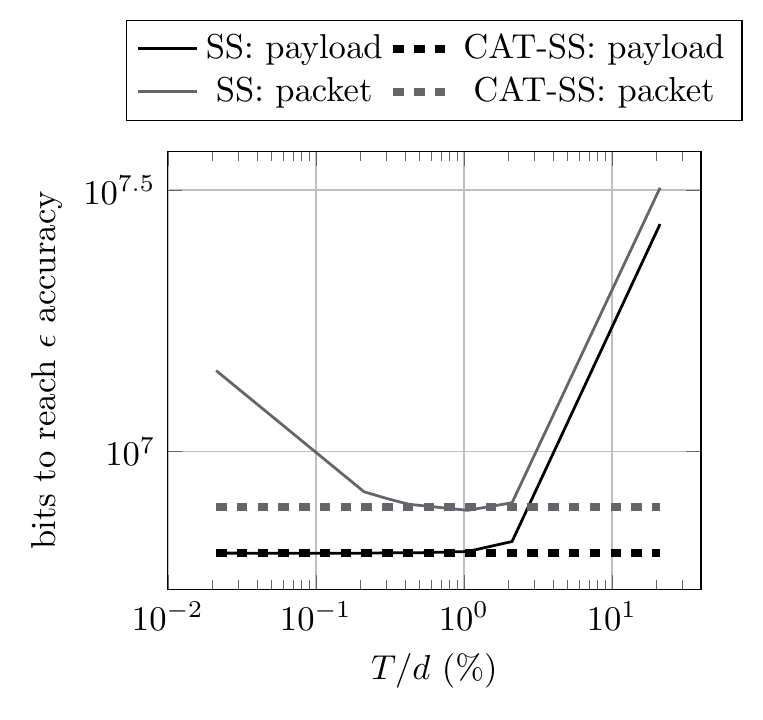}
    \caption{Expected communicated bits to reach $\epsilon$-accuracy for gradient descent with stochastic sparsification.}
    \label{fig:SS_singleNode}
\end{figure}

Next, we evaluated the performance of gradient descent with CAT stochastic sparsification. The communications are averaged over three Monte Carlo runs. Figure~\ref{fig:SS_singleNode} shows that stochastic sparsification has the same conclusions as deterministic sparsification. 
In the simplest payload model it is best to choose $T$ small.  
%
%
%
%
%
However, for the packet model  we need to  carefully tune  $T$ so that it  is neither to big or to small. Our CAT rule adaptively finds the best value of $T$ in both cases.

\end{document}